\pdfoutput=1
\documentclass[12pt,reqno,a4paper]{amsart}

\usepackage[margin=1.05in,footskip=0.15in,tmargin=1in,bmargin=0.85in]{geometry}
\usepackage[indent,skip=.2\baselineskip]{parskip}
\usepackage[utf8]{inputenc}
\usepackage[english]{babel}
\usepackage{amsthm,amsmath,amssymb}
\usepackage{mathtools}
\usepackage{amsfonts}
\usepackage{enumitem}
\usepackage{graphicx}
\usepackage[font=footnotesize]{caption}
\usepackage{algpseudocode}
\usepackage{floatrow}
\usepackage{blkarray}
\usepackage{tikz-cd}
\usepackage{chemfig}
\usepackage{array}
\usepackage{fancyhdr}
\usepackage{comment}
\usepackage[version=4]{mhchem}
\usepackage[dvipsnames]{xcolor}
\usepackage[ruled,vlined]{algorithm2e}
\usepackage{aliascnt}

\usepackage{accents}

\DeclareSymbolFont{yhlargesymbols}{OMX}{yhex}{m}{n}

\newcommand{\IF}{I(F)}

\newcommand{\IFk}{I(F_{u})}
\newcommand{\EFilt}{(E_{\rho}^{\bullet})_{\rho \in \Sigma(1)}}
\newcommand{\EFiltw}{(\widetilde{E}_{\rho}^{\bullet})_{\rho \in \Sigma(1)}}
\newcommand{\EFiltl}{(E(\ell)_{\rho}^{\bullet})_{\rho \in \Sigma(1)}}

\usepackage[unicode,psdextra,pdfusetitle,
bookmarks, bookmarksdepth=2, colorlinks=true,
linkcolor=blue!50!black, citecolor=black!70, urlcolor=blue!50!black]{hyperref}
\usepackage[capitalise, noabbrev,nameinlink]{cleveref}

\newcommand{\R}{\mathbb{R}}
\newcommand{\C}{\mathbb{C}}
\newcommand{\Tn}{\mathbb{T}^n}
\newcommand{\f}{\varphi}
\newcommand{\Ew}{\widetilde{E}}

\newcommand{\Cl}{\operatorname{Cl}}

\newcommand{\Nef}{\operatorname{Nef}}
\newcommand{\MV}{\mathbf{MV}}
\newcommand{\Pic}{\operatorname{Pic}}
\newcommand{\ch}{\text{ch}}

\newcommand{\Td}{\text{Td}}
\newcommand{\Div}{\operatorname{Div}}
\newcommand{\CDiv}{\operatorname{CDiv}}

\theoremstyle{definition}
\newtheorem{theorem}{Theorem}[section]

\newaliascnt{proposition}{theorem}
\newtheorem{proposition}[proposition]{Proposition}
\aliascntresetthe{proposition}
\newaliascnt{lemma}{theorem}
\newtheorem{lemma}[lemma]{Lemma}
\aliascntresetthe{lemma}
\newaliascnt{corollary}{theorem}
\newtheorem{corollary}[corollary]{Corollary}
\aliascntresetthe{corollary}
\newaliascnt{definition}{theorem}
\newtheorem{definition}[definition]{Definition}
\aliascntresetthe{definition}
\newaliascnt{example}{theorem}
\newtheorem{example}[example]{Example}
\aliascntresetthe{example}
\newtheorem{example*}{Example}

\newaliascnt{remark}{theorem}
\newtheorem{remark}[remark]{Remark}
\aliascntresetthe{remark}
\newaliascnt{conjecture}{theorem}

\aliascntresetthe{conjecture}
\newaliascnt{question}{theorem}

\aliascntresetthe{question}
\newaliascnt{assumption}{theorem}

\aliascntresetthe{assumption}
\newaliascnt{notation}{theorem}

\aliascntresetthe{notation}

\crefname{theorem}{theorem}{theorems}
\crefname{construction}{construction}{constructions}
\crefname{proposition}{proposition}{propositions}
\crefname{lemma}{lemma}{lemmas}
\crefname{corollary}{corollary}{corollaries}
\crefname{definition}{definition}{definitions}
\crefname{example}{example}{examples}
\crefname{remark}{remark}{remarks}
\crefname{conjecture}{conjecture}{conjectures}
\crefname{question}{question}{questions}
\crefname{assumption}{assumption}{assumptions}
\crefname{notation}{notation}{notations}

\newtheorem*{theorem*}{Theorem}
\newtheorem*{proposition*}{Proposition}
\newtheorem*{definition*}{Definition}

\newtheorem{theoremalphabetic}{Theorem}

\crefname{theoremalphabetic}{Theorem}{Theorems}

\DeclareMathOperator{\conv}{conv}

\DeclareMathOperator{\spann}{span}

\numberwithin{equation}{section}

\DeclareTextFontCommand{\bfemph}{\bfseries\em}

\newcommand{\restr}[2]{{\left.\kern-\nulldelimiterspace #1 \vphantom{\big|} \right|_{#2}}}

\renewcommand{\P}{\mathbb{P}}

\providecommand{keywords}[1]{
  \small
  \textbf{\textit{Keywords: }} #1
}

\newcommand{\Addresses}{{
  \bigskip
  \footnotesize
  \textsc{Carles Checa, University of Copenhagen}\par\nopagebreak
  \textit{E-mail address}, \texttt{ccn@math.ku.dk}
}}

\begin{document}

\title{Toric vector bundles \\ and the number of zeros of vertical systems}

\author{Carles Checa}

\maketitle

\begin{abstract}
We provide necessary and sufficient conditions for the specializations of a vertically parametrized polynomial system to attain the maximal number of complex nonzero solutions. To each vertical system, we attach a pair consisting of a projective simplicial toric variety and a toric vector bundle. The toric vector bundle allows for a homogenization of the polynomials in the Cox ring of the toric variety, providing a homogeneous ideal with the same zeros over the torus as the original system. We show that the maximal number of isolated solutions is attained if and only if this ideal has no solutions in the faces of the toric variety and prove that this happens for generic values of the parameters. In addition, we provide a novel formula for this generic number of zeros as an alternating sum of mixed volumes over a family of polytopes attached to the toric vector bundle.

\end{abstract}

\vspace{0.2cm}

\section{Introduction}

Exploiting sparsity is a central idea in computational and enumerative algebraic geometry. Namely, the geometry of polynomial systems over the torus depends on the combinatorics of the Newton polytopes of the distinct polynomial equations. Perhaps the most relevant result about the geometry of sparse systems is the \emph{Bernstein-Khovanskii-Kushnirenko} bound.

\begin{theorem*}\cite[Theorem 5.4]{coxlitosh}
\label{bkk}
Given Laurent polynomials $F_1,\dots,F_n$ with coefficients in $\mathbb{C}$ and Newton polytopes $\Delta_1,\dots,\Delta_n \subset \mathbb{R}^n$:
\begin{enumerate}[label = (\roman*)]
    \item
    \label{common_zeros}
    The number of isolated zeros of the system
    $$F_1 = \dots = F_n = 0$$ in $(\mathbb{C}^*)^n$ is bounded above by the (normalized) mixed volume $\MV(\Delta_1,\dots,\Delta_n)$.
    \item \label{generic_zeros}
    For generic values of the coefficients of the $F_i$, the bound is attained.
\end{enumerate}
\end{theorem*}

Toric geometry provides an interpretation of this bound: if we consider $X_{\Sigma}$ to be the projective toric variety given by the normal fan $\Sigma$ of the Minkowski sum $\Delta = \sum_{i=1}^n \Delta_i$, then each of the hypersurfaces $\{F_i = 0\}$ (and their Newton polytopes) can be associated with a nef divisor $D_i$ over $X_{\Sigma}$, which allows for a homogenization of the polynomials $F_i$ to global sections of $\mathcal{O}(D_i)$. The mixed volume can be recovered as the intersection product of the divisors $D_i$ (see \cite[Section 5.4]{toricfulton}), i.e.
\begin{equation}
\label{eq: mixedVolume}
(D_1 \cdot \: \dots \: \cdot D_n) = \MV(\Delta_1,\dots,\Delta_n).
\end{equation}
As the torus $\Tn \simeq (\mathbb{C}^*)^n$ is a dense open orbit in $X_{\Sigma}$, \eqref{eq: mixedVolume} explains the first part of the theorem. For the second part, the description of the coefficients for which the bound is attained is sometimes known as \textit{Bernstein's second theorem} \cite[Theorem B]{Bernshtein1975TheNO}. Assuming that the coefficients of each of the $F_1,\dots,F_n$ are algebraically independent parameters, the restrictions of the system to the faces of the toric variety are overdetermined polynomial systems which, for generic values of the coefficients, have no solutions. The locus of coefficients for which none of these restricted systems have a solution determines the polynomial systems for which the bound is attained.

In practice, polynomial systems usually have solutions in the faces of the toric variety determined by their Newton polytopes. The same phenomenon also occurs for distinct parametric families arising in applications, where the number of zeros for generic parameter values is strictly smaller than the mixed volume of the associated Newton polytopes. This observation motivates a more refined analysis of the possible combinatorial structures underlying the dependencies among the coefficients.

\emph{Vertically parametrized} polynomial systems (or \emph{vertical systems} for short) constitute an example of a natural family of systems for which the bound given by the mixed volume is often not attained for generic values of the parameters. These systems have arisen from the study of chemical reaction networks \cite{dickenstein2016biochemical, feinberg-book,feliu2025genericgeometrysteadystate}, as well as the critical equations of fewnomial hypersurfaces \cite{Ferrer:SONC}, and they are defined by a set of polynomials of the form
\begin{equation}
\label{verticalSystemIntro}
    F = C \big( u \star x^B \big) \ \in \C[u,x^{\pm}]^{s}
\end{equation}
with parameters $u=(u_1,\dots,u_{m})$ and variables $x=(x_1,\dots,x_n)$. Here, $C \in \C^{s \times m}$ is a full rank matrix with columns $\gamma_1,\dots,\gamma_m \in \mathbb{C}^s$, $B \in \mathbb{Z}^{n \times m}$ is a matrix whose columns $b_1,\dots,b_m  \in \mathbb{Z}^n$ correspond to the exponents of the monomials of the system and $u \star x^B$ indicates that
the $i$-th monomial is scaled by $u_i$. A particular feature of vertically parametrized systems is that each parameter accompanies the same monomial in all polynomials. We denote by $F_{u} \subset \C[x^{\pm}]^s$ each of the specializations of the systems for $u \in \C^m$. The terms \textit{engineered complete intersections} \cite{esterov2024engineeredcompleteintersectionsslightly, esterov2025engineeredcompleteintersectionseliminating} or \textit{vector-valued polynomial equations} \cite{kaveh2025vectorvaluedlaurentpolynomialequations} have also been used to describe these families. \textit{Sparse systems} (also called freely parametrized), where each of the parameters only appears in one of the equations, is a particular instance of vertically parametrized polynomial systems. A key distinction between general vertical systems and the subfamily of sparse systems is that the restriction of the system to the faces may fail to be overdetermined.

\begin{example*}
\label{example: Intro}
Consider the vertical family defined by the matrices
\[
C = \begin{pmatrix}
    1 & 0 & 1 & 1 & 1 & 1 \\
    0 & 1 & 1 & 2 & 1 & 1
\end{pmatrix}, \qquad
B = \begin{pmatrix}
    0 & 0 & 1 & 0 & 1 & 0 \\
    0 & 0 & 0 & 1 & 1 & 1
\end{pmatrix},
\]
which corresponds to the polynomial system
\begin{equation}
\label{first_system}
\begin{cases}
u_1 + u_3 x + u_4 y + u_5 x y + u_6 y = 0,\\
    u_2 + u_3 x + 2u_4 y + u_5 x y + u_6 y = 0.
        \end{cases}
\end{equation}
The Newton polytopes of both of the equations are $$\Delta = \conv\{(0,0),(1,0),(0,1),(1,1)\}.$$ Restricting to the face of the Newton polytopes given by the convex hull of $\{(1,0),(1,1)\} \subset \Delta$ gives the system
$$u_3x + u_5xy = 0.$$
This implies that the restricted system always has a solution. As a consequence, the generic root count over $(\mathbb{C}^*)^n$ cannot be equal to the mixed volume of the two polynomial equations.
\end{example*}

In some cases, a linear transformation in the rows of $C$ rectifies the support of each equation, avoiding this phenomenon. In other cases, all linear transformations of the matrix $C$ lead to a system whose restriction to some face is not overdetermined, leading to solutions at infinity for generic parameter values (see \Cref{ex: TangentExample}).

The generic number of zeros of vertical systems in $(\mathbb{C}^*)^n$ has been expressed in previous works using tropical geometry \cite{genericrootcounts,tropicalrootbounds} and also as the mixed volume over a family of virtual polytopes  \cite{kaveh2025vectorvaluedlaurentpolynomialequations}. Different algorithms to compute the zero locus of these systems have been described in \cite{esterovyuliarafael,helmincktropical}.
In this article, we describe necessary and sufficient conditions for these bounds to be attained, in an analogous way to Bernstein's second theorem.

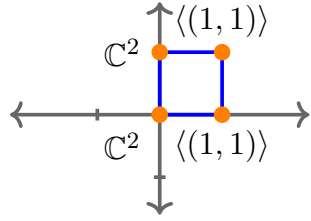
\begin{figure}
    \centering
\scalebox{1.1}{ \begin{tikzpicture}[x=0.75cm, y=0.75cm, line width=1.25pt]

      \foreach \x in {-1,1} {
        \draw[color=white!40!black] (\x, 2pt) -- (\x, -2pt);}
      \foreach \y in {-1,1} {
        \draw[color=white!40!black] (2pt, \y) -- (-2pt, \y);}
      \draw[color=white!40!black, <->] (-2.4, 0) -- (2.4,0) {};
      \draw[color=white!40!black, <->] (0, -1.6) -- (0, 1.8) {};

      \draw[color=blue, opacity=1.0] (0,0) -- (1,0) -- (1,1) -- (0,1) --
      (0,0) -- cycle;

      \node () at (1,-0.5) {\small $\langle (1,1) \rangle  $};
      \node () at (-0.6,1)  {\small $\mathbb{C}^2$};
      \node () at (-0.6,-0.5) {\small $\mathbb{C}^2$};
      \node () at (1,1.5)  {\small $\langle (1,1) \rangle $};

     \node[circle, fill=orange, inner sep=2.0pt] () at (0,0) {};
      \node[circle, fill=orange, inner sep=2.0pt] () at (1,0) {};
      \node[circle, fill=orange, inner sep=2.0pt] () at (0,1) {};
      \node[circle, fill=orange, inner sep=2.0pt] () at (1,1) {};
    \end{tikzpicture}}
    \caption{The information of a vertical system can be encoded as a set of pairs given by the lattice points $b \in \mathbb{Z}^n$ in the columns of $B$ and the vector spaces spanned by the columns of $\gamma_i$ of $C$ such that $b_i = b$. In the figure, the representation of the vertical system in \eqref{first_system}.}
    \label{fig: exampleIntro}
\end{figure}

\subsection*{Toric vector bundles}

The idea of using toric vector bundles to describe the zeros of vertical systems has recently arisen in the works of Kaveh, Khovanskii and Spink \cite{kaveh2025vectorvaluedlaurentpolynomialequations}. For the sparse case, these vector bundles are the direct sum of the line bundles $\mathcal{O}(D_i)$, where the $D_i$ are the divisors associated to the Newton polytopes of each of the polynomial equations. In the case of sparse systems, these divisors are useful either to compute the zeros of polynomial systems using eigenvalue methods \cite{bender2021toric} or to study sparse resultants \cite{BUSE2024107739, dandrea2025sparse}.

The advantage of working with toric vector bundles over a toric variety $X_{\Sigma}$ is that their information can be encoded using  \emph{Klyachko filtrations} (see \cite{klyachko}). These are decreasing filtrations of vector spaces $\EFilt$,  labeled by the rays of the fan $\Sigma$ (see \Cref{def: KlyachkoFiltration} for their definition). Recent works on the combinatorics of toric vector bundles open new doors to the generalization of some of the properties of sparse polynomial systems to other families (see \cite{parliament, altmann2024toricsheavespolyhedra, altmann2025acyclictoricsheaves, toricvectrobundles, multigradedregularityklyachko}).

\subsection*{Main results} Our main contribution consists of attaching to each vertical system $F = C(u \star x^B)$ an ideal $\IFk \subset R$ in the homogeneous coordinate ring of a toric variety, in a manner that the zeros of $\IFk$ over $\Tn$ coincide with the original zeros of $F_u$ and for generic values of the parameters, all zeros of $\IFk$ are in $\Tn$. To this end, we consider the polytope $\Delta_F \subset \mathbb{R}^n$ defined as:
$$\Delta_F \coloneqq \conv \Big\{\sum_{i \in \mathcal{I}}b_i \in \mathbb{Z}^n : \:  \mathcal{I} \subset [m], \: |\mathcal{I}| = s, \: (\gamma_i)_{i \in \mathcal{I}} \text{ are linearly independent}\Big\} \subset \R^n$$
and $\Sigma_F$ the normal fan of $\Delta_F$. For sparse systems, $\Delta_F$ coincides with the Minkowski sum of the Newton polytopes of the distinct equations.

We say that $(X_{\Sigma},\mathcal{E}_F)$ is a \emph{vertical pair} associated to $F$ (see \Cref{def: verticalVectorBundle}) if $X_{\Sigma}$ is a toric variety given by a fan $\Sigma$ and $\mathcal{E}_F$ is a toric vector bundle over $X_{\Sigma}$ satisfying:
\begin{enumerate}[label = (\roman*)]
    \item $X_{\Sigma}$ is projective, $\Sigma$ refines $\Sigma_F$ and contains a smooth maximal cone.

    \item The Klyachko filtration associated to $\mathcal{E}_F$ equals
        \begin{equation}
\label{eq: klyachko_data_vps_intro}
    E^{i}_{\rho} \coloneqq \spann \langle \gamma_j : \:   \langle v_{\rho}, b_j \rangle \leq -i \rangle \subset \mathbb{C}^s \quad \rho \in \Sigma(1) \text{ and } i \in \mathbb{Z},
\end{equation}
where $v_{\rho} \in \mathbb{Z}^n$ is the primitive generator of the ray $\rho \in \Sigma(1)$.
\end{enumerate}
We say that the vertical pair is \textit{simplicial} if the variety $X_{\Sigma}$ is simplicial. The structure of the Klyachko filtration allows us to describe a homogenization of the vertical system to a tuple of Laurent polynomials $\widetilde{F}_u \in \mathbb{C}^s \otimes R^{\pm}$ for each $u \in \mathbb{C}^m$, where $R = \C[z_{\rho} : \: \rho \in \Sigma(1)]$ is the homogeneous coordinate ring of $X_{\Sigma}$ (see \Cref{def: sectionhomgenization}). Although the tuple of Laurent polynomials $\widetilde{F}_u$ need not be globally defined in $X_{\Sigma}$, we show that there exists a nef divisor $D_F \in \Div(X_{\Sigma})$ that allows for the construction of a homogeneous ideal
\begin{equation}
\label{eq: IFKINTRO}
\IFk \coloneqq \big\langle  \varphi \otimes z^m(\widetilde{F}_u) : \: \varphi \otimes z^m \in H_{\mathcal{E}_F^{\vee}(D_F)}\big\rangle \subset R,
\end{equation}
for each parameter value $u \in \mathbb{C}^m$ (see \eqref{def: ideal}). Here, $H_{\mathcal{E}_F^{\vee}(D_F)} \subset (\mathbb{C}^s)^{\vee} \otimes R$ is a vector space that is isomorphic to the global sections of $\mathcal{E}_F^{\vee}(D_F)$ (see \Cref{thm: globalSections}). Using the homogeneous ideal $\IFk$, we give the following description of the set of isolated zeros of $F_u$, which we denote as $V_0(F_u)$.

\begin{theoremalphabetic}
\label[theoremalphabetic]{mainTheoremIntro}
Let $F = C(u \star x^B)$ be a vertical system with $s = n$ and let $(X_{\Sigma},\mathcal{E}_F)$ be a simplicial vertical pair associated to it.
\begin{enumerate}[label = (\roman*)]
    \item(\Cref{thm: Torus}) For each $u \in \C^m$, the subschemes of $\Tn$ defined by $I(F_{u})$ and $F_{u}$ are isomorphic.
    \item(\Cref{thm: MainNoSolutionsInfinity}) There exists an open subset $U \subset \C^m$ such that:
    \begin{equation}
        V(\IFk) \cap V(z_{\rho}) = \varnothing, \quad \forall \rho \in \Sigma(1), \: \forall u \in U.
    \end{equation}
    \item(\Cref{cor: numberOfIsolatedSolutions}) For each $u \in \mathbb{C}^m$, we have
    $$|V_0(F_u) \cap \Tn| = \max_{u' \in \mathbb{C}^m}|V_0(F_{u'}) \cap \Tn|,$$
    if and only if $V(\IFk) \cap V(z_{\rho}) = \varnothing$ for all $\rho \in \Sigma(1)$.
\end{enumerate}
\end{theoremalphabetic}

All in all, \Cref{mainTheoremIntro} gives an analogue of Bernstein's second theorem for vertical systems: the generic number of zeros over $\Tn$ of the vertical system $F = C(u \star x^B)$ is attained if and only if the ideal $\IFk$ has no zeros in the faces of the toric variety $X_{\Sigma}$, which happens in an open subset of the parameter space.

In the case where the generic number of solutions is nonzero, which has been called \textit{generically consistent} in \cite{FELIU2025630}, we express it in terms of the degree of the top Chern class of $\mathcal{E}_F$ (see \Cref{thm: VerticalHirzebruchRiemannRoch}), which allows for a reinterpretation of the generic root count in terms of the mixed volumes of a family of polytopes associated to $\mathcal{E}_F$.

Namely, each toric vector bundle has a naturally attached family of Weil divisors, which we denote as $\mathcal{S}_E = \{D_1,\dots,D_t\} \subset \Div(X_{\Sigma})$, which can be described using the Klyachko filtration \cite{altmann2024toricsheavespolyhedra, KHAN2026110646, parliament} (see \eqref{eq: SE}).  We label each of the divisors in $\mathcal{S}_E$ with an integer
$\delta_{i} \in \mathbb{Z}$ for $i = 1,\dots,t$ (see \eqref{rk: delta_D}) and use them to provide a new formula for the generic number of isolated zeros of vertical systems.

\begin{theoremalphabetic}[\Cref{thm: RootCountFinalFormula}]
\label[theoremalphabetic]{thmc}
Let $F = C(u \star x^B)$ be a generically consistent vertical system with $s = n$ and let $(X_{\Sigma},\mathcal{E}_F)$ be a simplicial vertical pair. For each $u \in \C^m$, we have
$$|V_0(F_{u}) \cap \Tn| \leq \sum_{\substack{p \in \mathbb{Z}_{\geq0}^{t} \\
p_1 + \dots + p_t = n}} \;
\mu_{p}\MV\big( \Delta_1[p_1],\dots,\Delta_t[p_t]\big), \quad \mu_{p} \coloneqq
    \prod_{i = 1}^t\binom{\delta_i}{p_i}, $$
where $\Delta_1,\dots,\Delta_t \subset \mathbb{R}^n$ are the polytopes associated with the divisors $D_1,\dots,D_t$ and $\Delta[p]$ indicates $p$ copies of the polytope $\Delta$. Moreover, the bound is attained if and only if $\IFk$ has no solutions at infinity, which happens in an open subset of $\mathbb{C}^m$. \footnote{If some of the divisors $D \in \mathcal{S}_E$ are not nef, the resulting polytope will be virtual, i.e. a formal difference of polytopes $\Delta - \Delta'$. In this setting, the mixed volume is understood by extending it multilinearly.}
\end{theoremalphabetic}

We show with examples that the above formula can be reduced to an easy computation with the Klyachko filtration of $\mathcal{E}_F$ (see \Cref{ex: continueTangentBundle}).
In some cases, the generic root count can still be reduced to the computation of a single mixed volume.

\begin{theoremalphabetic}[\Cref{thm: splitRootCount}]
\label[theoremalphabetic]{theo: mainTheoremSplit}
Let $F = C(u \star x^B)$ be a vertical system and $(X_{\Sigma},\mathcal{E}_F)$ a vertical pair associated to $F$ and let $u \in \mathbb{C}^m$. If $s = n$ and $\mathcal{E}_F = \bigoplus_{i = 1}^n \mathcal{O}(D_i)$ for some Cartier divisors $D_i$, then
$$|V_0(F_{u}) \cap \Tn| \leq \MV(\Delta_1,\dots,\Delta_n), $$
where $\Delta_i$ are the polytopes associated with $D_i$ for each $i \in \{1,\dots,s\}$ (see \eqref{eq:polytope}). Moreover, the bound is attained if and only if $\IFk$ has no solutions at infinity, which happens in an open subset of $\mathbb{C}^m$.
\end{theoremalphabetic}

The paper is structured as follows: in \Cref{section: preliminaries}, we review the basic results on toric vector bundles and intersection theory needed for our proofs. In \Cref{section: vertical}, we describe vertical systems and define vertical pairs $(X_\Sigma,\mathcal{E}_F)$ associated to them. With this, we provide the construction of the homogeneous ideal $\IFk \subset R$ and prove \cref{mainTheoremIntro} which relates the zeros of $\IFk$ with the zeros of the family $F$ over $\Tn$ and shows that $\IFk$ is an overdetermined polynomial system in the faces of $X_{\Sigma}$. In \cref{sec: KoszulChernZeros}, we first reduce the computation of the generic root count of square vertical families to computing the Chern classes of $\mathcal{E}_F$. Then, we prove \Cref{thmc} and \Cref{theo: mainTheoremSplit} via explicitly computing the Chern classes of $\mathcal{E}_F$, providing novel formulas for the generic root count over $\Tn$ of vertical systems.

\subsection*{Acknowledgments}
This project has been funded by the European Union under the Grant Agreement number 101044561, POSALG. Views and opinions expressed are those of the authors only and do not necessarily reflect those of the European Union or European Research Council (ERC).

\section{Preliminaries on toric vector bundles}
\label{section: preliminaries}

In this section, we present the definitions and results that are necessary for the rest of the paper. These include toric varieties, the relation between polytopes and nef divisors, toric vector bundles and Klyachko filtrations and their properties. The main references for these contents are \cite{coxlittleschneck, toricvectrobundles, parliament, altmann2024toricsheavespolyhedra}. In addition, we prove some auxiliary results on toric vector bundles (for instance, \Cref{prop: maxminwedgeDual}), which can be of interest beyond the context of polynomial systems.

\subsection{Toric varieties}
Let \(M\) be a lattice of rank \(n\) (so \(M \simeq \mathbb{Z}^n\)). We denote by \(N = \operatorname{Hom}_{\mathbb{Z}}(M, \mathbb{Z})\) the dual lattice and by \(\Tn = N \otimes_{\mathbb{Z}} \mathbb{C}^{\times}\) the associated algebraic torus. We set \(M_{\mathbb{R}} = M \otimes_{\mathbb{Z}} \mathbb{R}\) , \(N_{\mathbb{R}} = N \otimes_{\mathbb{Z}} \mathbb{R}\) and \(M_{\mathbb{Q}} = M \otimes_{\mathbb{Z}} \mathbb{Q}\).

Let \(\Sigma\) be a complete fan in \(N_{\mathbb{R}}\) and let \(X_{\Sigma}\) be the corresponding complete toric variety. For a cone \(\sigma \in \Sigma\), let \(U_{\sigma}\) be the associated affine \(\Tn\)-invariant open subset, via the orbit-cone correspondence \cite[Theorem 3.2.6]{coxlittleschneck}. These open subsets form a cover of $X_{\Sigma}$, i.e.
\begin{equation}
\label{eq:opencover}
X_{\Sigma} = \bigcup_{\sigma \in \Sigma(n)}U_{\sigma}.
\end{equation}

The toric variety $X_{\Sigma}$ is \textit{smooth} if, for every cone $\sigma \in \Sigma$, the primitive generators of its rays can be extended to a $\mathbb{Z}$-basis of $N$. It is \textit{simplicial} if, for every cone $\sigma \in \Sigma$, the primitive generators of its rays are linearly independent in $N_{\mathbb{R}}$. For each ray \(\rho \in \Sigma(1)\), let \(v_{\rho} \in N\) be its primitive generator. For a cone \(\sigma \in \Sigma\), we define the quotient lattice
\[
M_{\sigma} \coloneqq M/(\sigma^{\perp} \cap M),
\]
where \(\sigma^{\perp} = \{ m \in M_{\mathbb{R}} \mid \langle v, m \rangle = 0 \text{ for all } v \in \sigma \}\).

Consider the short exact sequence
\begin{equation}
\label{eq:ses}
    0 \longrightarrow M \stackrel{\mathbf{V}}{\longrightarrow} \mathbb{Z}^{\Sigma(1)} \stackrel{\pi}{\longrightarrow} \Cl(X_{\Sigma}) \longrightarrow 0,
\end{equation}
where \(\mathbf{V}\) is the matrix whose rows are the ray generators \(v_{\rho}\) and \(\pi\) is a Gale dual map. The \emph{Cox ring} (or homogeneous coordinate ring) of the toric variety $X_{\Sigma}$ is the polynomial ring \(R = \C[z_{\rho} \mid \rho \in \Sigma(1)]\) with the \(\Cl(X_{\Sigma})\)-grading induced by \(\pi\), i.e.
\begin{equation}
\label{eq: CoxRing}
R = \bigoplus_{\alpha \in \Cl(X_{\Sigma})} R_{\alpha} \cong \bigoplus_{[D] \in \Cl(X_{\Sigma})} H^0(X_{\Sigma},\mathcal{O}(D)).
\end{equation}
The $\Cl(X_{\Sigma})$-grading in $R$ is provided by the map $\pi$ in \eqref{eq:ses}. If $z^{\widehat{\sigma}} \coloneqq \prod_{\rho \notin \sigma(1)}z_{\rho}$, then the coordinate ring of $U_{\sigma}$ is the degree-zero part of the localization, namely $\mathbb{C}[U_{\sigma}] \cong (R_{z^{\widehat{\sigma}}})_0$, which we denote as $R_{\sigma}$. In particular, the variables $z_{\rho}$ for $\rho \notin \sigma(1)$ are invertible in $R_{\sigma}$.

The divisors $D_{\rho} = V(z_{\rho})$ are invariant under the action of $\Tn$ on $X_{\Sigma}$. Their integer linear combinations are called \emph{$\Tn$-invariant Weil divisors}, and we denote their group by $\Div(X_{\Sigma})$. We similarly speak of $\mathbb{Q}$-Weil and $\mathbb{R}$-Weil divisors when allowing coefficients in $\mathbb{Q}$ and $\mathbb{R}$, respectively.

\subsection{Cartier and nef line bundles}

We say that a \(\Tn\)-invariant Weil divisor \(D = \sum_{\rho} a_{\rho} D_{\rho}\) is \emph{effective} if $a_{\rho}\geq 0$ for all $\rho \in \Sigma(1)$. With this, we define the poset relation on $\Tn$-invariant divisors as
\begin{equation}
\label{eq: poset}
D \leq D' \iff D' -D \text{ is effective} \end{equation}
i.e. $a_{\rho} \leq a'_{\rho}$ for all $\rho \in \Sigma(1)$. Moreover, given two divisors $D = \sum_{\rho \in \Sigma(1)}a_{\rho}D_{\rho}$ and $D' = \sum_{\rho \in \Sigma(1)}a'_{\rho}D_{\rho}$, we can define
\begin{equation}
\label{eq: minmax}
\max(D_1,D_2) = \sum_{\rho \in \Sigma(1)}\max(a_{\rho},a_{\rho}')D_{\rho} \quad \min(D_1,D_2) = \sum_{\rho \in \Sigma(1)}\min(a_{\rho},a_{\rho}')D_{\rho}.
\end{equation}

\begin{definition}
\label[definition]{def: Cartier}
The divisor $D$ is \textit{Cartier} (resp. $\mathbb{Q}$-Cartier) if for all $\sigma \in \Sigma(n)$ there exists $m_{\sigma} \in M$ (resp. $m_{\sigma} \in M_{\mathbb{Q}}$) such that
\begin{equation}
\label{eq: cartierData}
\langle v_{\rho},m_{\sigma} \rangle + a_{\rho} = 0, \quad \forall \rho \in \sigma(1).
\end{equation}
We denote by $\CDiv(X_{\Sigma}) \subset \mathbb{Z}^{\Sigma(1)}$ the set of Cartier divisors and their
classes by \eqref{eq:ses} form the Picard group $\Pic(X_{\Sigma})$.
\end{definition}

In a simplicial toric variety, every Weil divisor is $\mathbb{Q}$-Cartier \cite[Proposition 4.2.7]{coxlittleschneck}. Each $\Tn$-invariant Weil divisor \(D\) can be associated with a polytope, i.e.
\begin{equation}
\label{eq:polytope}
    \Delta_D = \{\, m \in M_{\mathbb{R}} \mid \langle v_{\rho}, m \rangle \geq -a_{\rho} \text{ for all } \rho \in \Sigma(1) \,\}.\footnote{We are following the convention for attaching a polytope to a divisor appearing in \cite{coxlittleschneck}. In other articles, such as \cite{parliament}, the polytopes are described with the opposite sign on $v_{\rho}$.}
\end{equation}

The following characterization is useful to determine when the divisors $D$ are nef.
\begin{proposition}\cite[Theorem 6.1.7, Theorem 6.3.12]{coxlittleschneck}
\label{prop: criterionNef}
Let $X_{\Sigma}$ be a complete toric variety and let $D \in \Div(X_{\Sigma})$ be a Cartier divisor. The following are equivalent:
\begin{enumerate}[label = (\roman*)]
    \item $D$ is nef.
    \item $\mathcal{O}(D)$ is globally generated.
    \item For all $\sigma \in \Sigma(n)$ and $m_{\sigma} \in M$ as in \eqref{eq: cartierData}, we have $m_{\sigma} \in \Delta_D$.
\end{enumerate}
\end{proposition}

\subsection{Toric vector bundles}

A \emph{toric vector bundle} \(\mathcal{E}\) of rank $s$ on \(X_{\Sigma}\) is a vector bundle equipped with a \(\Tn\)-linearization, i.e., an algebraic action of \(\Tn\) on the total space of \(\mathcal{E}\) lifting the action on \(X_{\Sigma}\). We fix a point \(x_0\) in the dense torus orbit and denote the fiber over \(x_0\) by \(E \coloneqq \mathcal{E}_{x_0}\), which is a vector space of dimension \(s\).

On an affine toric variety \(U_{\sigma}\), every toric vector bundle splits equivariantly as a sum of toric line bundles. For \(u \in M\), let \(\mathcal{L}_u\) be the trivial line bundle on \(U_{\sigma}\), where $\Tn$ acts on $\mathbb{C}$ via the character $x^u: \Tn \xrightarrow[]{} \mathbb{C}^*$. The line bundle \(\mathcal{L}_u\) depends only on the class \([u] \in M_{\sigma}\). A toric vector bundle of rank $s$ on \(U_{\sigma}\) decomposes as
\begin{equation}
\label{eq: onUsigmaSplits}
\mathcal{E}|_{U_{\sigma}} = \bigoplus_{i=1}^{s} \mathcal{L}_{[u_i]}
\end{equation}
for uniquely determined classes \([u_1],\dots,[u_s] \in M_{\sigma}\). We denote this multiset by \(\mathbf{u}(\sigma) \subset M_{\sigma}\). A convenient way to encode the information in a toric vector bundle is through a \emph{compatible collection of filtrations}.

\begin{definition}
\label[definition]{def: KlyachkoFiltration}
    A collection of filtrations $\EFilt$ of a vector space \(E\) is \emph{compatible with the fan \(\Sigma\)} if it satisfies
    \begin{enumerate}[label = (\roman*)]
        \item There exists a decomposition \(E = \bigoplus_{[u] \in \mathbf{u}} L_{[u]}\) indexed by a finite multiset \(\mathbf{u} \subset M_{\sigma}\) for each cone \(\sigma \in \Sigma\), where each \(L_{[u]}\) is a one-dimensional subspace,
        \item For every maximal cone $\sigma \in \Sigma(n)$, every \(\rho \in \sigma(1)\) and every integer \(i\),
        \begin{equation}
        \label{eq: compatibility}
        E^{\,i}_{\rho} = \sum_{\substack{[u] \in \mathbf{u} \\ \langle v_{\rho}, u \rangle \geq i}} L_{[u]};
        \end{equation}
        \item $\dots \supset E^i_{\rho} \supset E^{i+1}_{\rho} \supset \dots$ for all $\rho \in \Sigma(1)$.
         \item \(\bigcap_{i \in \mathbb{Z}} E^{\,i}_{\rho} = \{0\}\) and \(\sum_{i \in \mathbb{Z}} E^{\,i}_{\rho} = E\) for all \(\rho \in \Sigma(1)\).
    \end{enumerate}
\end{definition}

The above data is known as a \emph{Klyachko filtration} and will be denoted as $(E_{\rho}^{\bullet})_{\rho \in \Sigma(1)}$, throughout the paper.

\begin{theorem}\cite{klyachko}
    The category of toric vector bundles on \(X_{\Sigma}\) is equivalent to the category of compatible collections of filtrations.
\end{theorem}

\subsection{The algebra of Klyachko filtrations}

 Throughout, we will work directly with Klyachko filtrations, in order to define toric vector bundles and combinatorial objects over them. The following operations are useful for the manipulations that we will do and their proofs can be found in \cite{dasguptadeykhan}.

\begin{enumerate}[label = (\roman*)]
    \item \textbf{Toric line bundles: } Given a $\Tn$-invariant divisor $D = \sum_{\rho \in \Sigma(1)}a_{\rho}D_{\rho}$, the toric line bundle $\mathcal{O}(D)$ has rank $1$ and associated Klyachko filtration
    \begin{equation}
    \label{eq: toricLineBundle}
    E_{\rho}^i = \begin{cases}
        \mathbb{C} & i \leq a_{\rho} \\
        0 & i > a_{\rho},
    \end{cases}
    \end{equation}
    for each $\rho \in \Sigma(1)$ and $i \in \mathbb{Z}$.
    \item \textbf{Tensor product vector bundle: } Let $\mathcal{E},\widetilde{\mathcal{E}}$ be toric vector bundles over $X_{\Sigma}$ represented by two Klyachko filtrations $\EFilt$ and $\EFiltw$, then the tensor product is a toric vector bundle with Klyachko filtration
    $$(E \otimes \Ew)_{\rho}^i =  \spann \langle w \otimes  w' : \: w \in E^{i_1}_{\rho}, \: w' \in \Ew^{i_2}_{\rho}, \: i_1 + i_2  \geq i \rangle $$
    In particular, if $\widetilde{\mathcal{E}} = \mathcal{O}(D)$ for some $\Tn$-invariant divisor  $D = \sum_{\rho \in \Sigma(1)}a_{\rho}D_{\rho}$, then
    \begin{equation}
    \label{eq: tensorwithaFiltration}
    (E \otimes \Ew)_{\rho}^i = E_{\rho}^{i - a_{\rho}}.
    \end{equation}
    \item \textbf{Dual vector bundle:} Let $\mathcal{E}$ be a toric vector bundle over $X_{\Sigma}$ represented by the Klyachko filtration $E^i_{\rho}$, then the dual vector bundle $\mathcal{E}^{\vee}$ is a toric vector bundle with filtration
    \begin{equation}
    \label{eq: dualFiltration}(E^{\vee})^i_{\rho} = (E_{\rho}^{1-i})^{\perp} = \{ \varphi \in E^{\vee} : \: \f(w) = 0, \: \forall w \in E_{\rho}^{1-i} \}.
    \end{equation}

    \item \textbf{Wedge vector bundle:} Let $\mathcal{E}$ be a toric vector bundle of rank $s$ over $X_{\Sigma}$ represented by the Klyachko filtration $\EFilt$ and let $\ell \in \{1,\dots,s \}$, then the $\ell$-th wedge vector bundle $\wedge^{\ell}\mathcal{E}$ is a toric vector bundle of rank $\binom{s}{\ell}$ with filtration
   \begin{equation}
   \label{eq: wedgeFiltration}
   (\wedge^{\ell}E)^i_{\rho} =  \langle w_1 \wedge \dots \wedge w_{\ell} : \: w_j \in E^{i_j}_{\rho}, \: j \in \{1,\dots,\ell\}, \: \sum_{j = 1}^{\ell} i_j \geq i \rangle .
   \end{equation}
   \item \textbf{Jump divisors: } To a toric vector bundle $\mathcal{E}$ of rank $s$ over $X_{\Sigma}$ with filtration $(E_{\rho}^{\bullet})_{\rho \in \Sigma(1)}$, one can attach a family of divisors, i.e. for $\ell \in \{ 1,\dots,s \}$, we attach the $\Tn$-invariant divisor
\begin{equation}
\label{eq: jump}
D_{\mathcal{E},\ell} = \sum_{\rho \in \Sigma(1)}\omega_{\rho,\ell}D_{\rho}, \quad \omega_{\rho,\ell} := \max ( i \in \mathbb{Z}: \dim E_{\rho}^i \geq \ell ).
\end{equation}

\end{enumerate}

The following property of jump divisors will be useful to describe our results.

\begin{proposition}
\label[proposition]{prop: maxminwedgeDual}
Let $\mathcal{E}$ be a toric vector bundle of rank $s$ over a toric variety $X_{\Sigma}$ with Klyachko filtration $\EFilt$. Then, for all $\ell \in \{1,\dots,s\}$, we have
 \begin{enumerate}[label = (\roman*)]
     \item $D_{\mathcal{E}^{\vee},\ell} = -D_{\mathcal{E},s-\ell + 1}$.
    \item $D_{\wedge^{\ell}\mathcal{E},1} = \sum_{j = 1}^{\ell} D_{\mathcal{E},j}$.
\end{enumerate}

\end{proposition}

\begin{proof}
We fix $\rho \in \Sigma(1)$ and note that:
\begin{align*}
\max( i \in \mathbb{Z} : \: \dim(E^{1-i}_{\rho})^{\perp} \geq \ell)  = \max (i \in \mathbb{Z} : \: \dim E_{\rho}^{1-i} \leq s - \ell) \hspace{0.2cm} \\
= 1-\min (i \in \mathbb{Z} : \: \dim E_{\rho}^{i} \leq s - \ell)  = -\max (i \in \mathbb{Z} : \: \dim E_{\rho}^{i} \geq s - \ell + 1).
\end{align*}
deducing (i).

We prove (ii) by induction on $\ell$. If $\ell = 1$, the result is trivial. When $\ell > 1$, we know that if $w_1,\dots,w_{\ell} \in E$ are such that $w_j \in E^{i_j}_{\rho}$ for some $i_1,\dots,i_{\ell} \in \mathbb{Z}$ with $\sum_{j = 1}^{\ell}i_j$ being maximal, then $\langle w_1,\dots,w_{\ell} \rangle \subseteq E^{\omega_{\rho,\ell}}_{\rho}$. Moreover, by \eqref{eq: wedgeFiltration}, we have
$$   (\wedge^{\ell}E)^i_{\rho} =  \langle w_1 \wedge \dots \wedge w_{\ell} : \: w_1 \wedge \dots \wedge w_{\ell - 1} \in (\wedge^{\ell - 1}E)^{i_1}_{\rho}, \: w_{\ell} \in E_{\rho}^{i_2}, \: i_1 + i_2 \geq i \rangle \subset \wedge^{\ell}E. $$
By induction, the largest value of $i = \sum_{j = 1}^{\ell - 1}i_1$ for which $0 \neq w_1 \wedge \dots \wedge w_{\ell - 1} \in (\wedge^{\ell - 1}E)^{i}_{\rho}$ equals $\sum_{j = 1}^{\ell - 1}\omega_{\rho,j}$ and
$$\langle w_1,\dots,w_{\ell-1} \rangle \subseteq E_{\rho}^{\omega_{\rho,\ell - 1}}.$$
Moreover, the largest possible $i_2$ such that there exists $w_{\ell} \in E_{\rho}^{i_2}$ with $w_1 \wedge \dots \wedge w_{\ell} \neq 0$ must be $\omega_{\rho,\ell}$ as, otherwise, $w_{\ell} \in \langle w_1,\dots,w_{\ell - 1} \rangle $. Therefore, the largest possible value of $i$ where $(\wedge^{\ell}E)_{\rho}^i \neq 0$ equals $\sum_{j = 1}^{\ell}\omega_{\rho,j}$.
\end{proof}

Finally, we recall the description of the global sections of a toric vector bundle in terms of the combinatorial information of its Klyachko filtration.

\begin{definition}
\label[definition]{relevant_linear_spaces}
Let $\mathcal{E}$ be a toric vector bundle over a toric variety $X_{\Sigma}$ with Klyachko filtration $(E_{\rho}^{\bullet})_{\rho \in \Sigma(1)}$. To each lattice point $b \in M$, one can attach a linear space
\begin{equation}
\label{eq: Vu}
V^b \coloneqq \bigcap_{\rho \in \Sigma(1)}E_{\rho}^{-\langle v_{\rho},b \rangle}.
\end{equation}
\end{definition}

The following theorem appeared in the work of Klyachko \cite{klyachko} and has also been used in \cite{parliament}.

\begin{theorem}
\label[theorem]{thm: globalSectionsVu}
Let $X_{\Sigma}$ be a toric variety and let $\mathcal{E}$ be a toric vector bundle over $\Sigma$ with Klyachko filtration $(E_{\rho}^{\bullet})_{\rho \in \Sigma}$. Then, there exists an isomorphism:
$$H^0(X_{\Sigma},\mathcal{E}) \cong \sum_{b \in M}\spann(e \otimes x^b : \: e \in V^b) \subset \mathbb{C}^s \otimes \C[x^{\pm}].$$
\end{theorem}

\subsection{Posets and resolutions}

A toric vector bundle $\mathcal
E$ over a toric variety $X_{\Sigma}$ can naturally be associated with a poset of Weil divisors, which can be described as follows: for each $e \in E$, consider the divisor $D_{\mathcal{E}}(e) \in \Div(X_{\Sigma})$ defined as

\begin{equation}
\label{eq: divisore}
D_{\mathcal{E}}(e) \coloneqq \sum_{\rho \in \Sigma(1)}a_{\rho,e}D_{\rho}, \quad a_{\rho,e} \coloneqq
\max (i \in \mathbb{Z} : \: e \in E_{\rho}^i),
\end{equation}
where $\EFilt$ is the Klyachko filtration of $\mathcal{E}$ and let
\begin{equation}
\label{eq: Deltae}
\Delta_{\mathcal{E}}(e) \coloneqq \{b \in M_{\mathbb{R}} : \: \langle v_{\rho},b \rangle \geq -a_{\rho,e} \quad \forall \rho \in \Sigma(1)\}.
\end{equation}
be the corresponding polytope.
We consider the poset $(\mathcal{S}_E,\leq)$ of all such divisors, i.e.
\begin{equation}
\label{eq: SE}
\mathcal{S}_E \coloneqq \{D_{\mathcal{E}}(e) \in \Div(X_{\Sigma}) : \: e \in E\},
\end{equation}
where $\leq$ is the relation induced in \eqref{eq: poset}. The set $\mathcal{S}_E$ is finite \cite[Proposition 3.2]{altmann2024toricsheavespolyhedra} and satisfies that
\begin{equation}
\label{eq: sum}
D_{\mathcal{E}}(e + e') \geq \min(D_{\mathcal{E}}(e),D_{\mathcal{E}}(e')).
\end{equation}
For each divisor $D \in \mathcal{S}_E$, define
\begin{equation}
V_D \coloneqq \{ e \in E : D_{\mathcal{E}}(e) \geq D \}.
\end{equation}
By \eqref{eq: sum}, each $V_D$ is a vector subspace of $E$.

The poset structure of $\mathcal{S}_E$ can be induced from the lattice of flats of a matroid. In \cite[Proposition 3.1]{parliament}, Di Rocco, Jabbusch and Smith gave a construction of a matroid whose lattice of flats equals the poset $\mathcal{S}_E$. Moreover, from the (generators of) $1$-dimensional flats $e_1,\dots,e_{\ell} \in E$ of this matroid, they constructed a \emph{parliament of polytopes} $(\Delta_1,\dots,\Delta_\ell)$ (as in \eqref{eq: Deltae}), which was used to generate the global sections of $\mathcal{E}$ (see \cite[Proposition 1.2]{parliament}), i.e.
\begin{equation}
    H^0(X_{\Sigma},\mathcal{E}) \cong \sum_{i = 1}^{\ell} \spann(e_i \otimes x^u : \: u \in \Delta_i \cap M) \subset \mathbb{C}^s \otimes \mathbb{C}[x^{\pm}].
\end{equation}
In \cite[Section 3]{KHAN2026110646}, Khan and Maclagan constructed a polymatroid from the data in the Klyachko filtration and described a unique smallest matroid inducing that polymatroid.

\begin{lemma}
\label[lemma]{lem: neflyDecoratedExists}
    Let $X_{\Sigma}$ be a projective simplicial toric variety and let $\mathcal{E}$ be a toric vector bundle over $X_{\Sigma}$. Then, there exists a nef Cartier divisor $D \in \Div(X_{\Sigma})$ such that all the divisors in the poset of $\mathcal{E} \otimes \mathcal{O}(D)$ are nef.
\end{lemma}

\begin{proof}
Let $\mathcal{S}_E = \{D_1, \dots, D_t\}$. Since $X_{\Sigma}$ is simplicial, every $D_i$ is $\mathbb{Q}$-Cartier. Let $D_A$ be an ample Cartier divisor on $X_{\Sigma}$. Recall that the ample cone is the interior of the nef cone (\cite[Theorem 6.4.9]{coxlittleschneck}). In particular, for each $i$, there exists $k_i \in \mathbb{Z}_{\geq 0}$ such that the divisor $D_i+kD_A$ is ample for $k \geq k_i$. Since $\mathcal{S}_E$ is finite, we may choose a single $k \gg 0$ such that $D_i+kD_A$ is ample, and hence nef, for every $i=1,\dots,t$.

Tensoring with $\mathcal{O}(kD_A)$ shifts the Klyachko filtration as in \eqref{eq: tensorwithaFiltration}, so
\[
D_{\mathcal{E} \otimes \mathcal{O}(kD_A)}(e) = D_{\mathcal{E}}(e) + kD_A
\]
for each $e \in E$. Hence all the divisors in the poset of $\mathcal{E} \otimes \mathcal{O}(kD_A)$ are nef.
\end{proof}

Altmann, Hochenegger and Witt \cite[Corollary 6.10]{altmann2024toricsheavespolyhedra} have recently provided a locally free resolution of $\mathcal{E}$ using the poset $\mathcal{S}_E$. In order to describe a resolution of $\mathcal{E}$, consider the set of strictly increasing chains of length $k \in \mathbb{Z}_{\geq 0}$ between two fixed divisors $D,D' \in \mathcal{S}_E$, i.e.
\begin{equation}
\label{eq:chains}
\mathcal{C}_k(D,D') \coloneqq \{ D = D_0 < D_1 < \dots < D_k = D' : \:  D_0,\dots,D_k \in \mathcal{S}_E\}.
\end{equation}

\begin{theorem}\cite[Corollary 6.10]{altmann2024toricsheavespolyhedra}
\label{thm: resolutions}
Let $X_{\Sigma}$ be a complete toric variety and let $\mathcal{E}$ be a toric vector bundle of rank $s$. There exists an exact complex, whose terms are
\begin{equation}
\label{eq: resolution}
0 \xrightarrow{} \bigoplus_{D,D' \in \mathcal{S}_E}\bigoplus_{\mathcal{C}_s(D,D')} \big(V_D \otimes \mathcal{O}(D')\big) \xrightarrow{} \dots \xrightarrow[]{} \bigoplus_{D,D' \in \mathcal{S}_E}\bigoplus_{\mathcal{C}_0(D,D')} \big(V_D \otimes \mathcal{O}(D')\big) \xrightarrow[]{} \mathcal{E} \xrightarrow[]{} 0.
\end{equation}
\end{theorem}

\begin{remark}
\label[remark]{rk: itIsIndeedSumOfLineBundles}
    The toric vector bundle $V_{D} \otimes \mathcal{O}(D)$ for $D' = \sum_{\rho \in \Sigma(1)}a_{\rho}D_{\rho}$ is defined by the Klyachko filtration
    $$ E_{\rho}^i = \begin{cases}
        V_D & i \leq a_{\rho} \\
        0 & i > a_{\rho}
    \end{cases}.$$
    In particular, it splits as $r$ copies of $\mathcal{O}(D')$ for $r = \dim(V_{D})$.
\end{remark}

\subsection{Hirzebruch-Riemann-Roch theorem}

In order to prove the generic root count provided in \Cref{thmc}, we will have to reduce the root count to the computation of the degree of the top Chern class of a toric vector bundle. In this section, we recall some aspects of intersection theory that will be needed to prove that theorem.  We refer to \cite[Chapter 13]{coxlittleschneck} and \cite{FultonSturmfels1997} for the results. We denote by $A^{\bullet}(X_{\Sigma})$ the Chow cohomology ring of $X_{\Sigma}$, whose combinatorial description can be found in \cite[Section 5]{FultonSturmfels1997}.

\begin{definition}
    Let $X_{\Sigma}$ be a complete toric variety and let $H^i(X_{\Sigma},\mathbb{Z})$ be its $i$-th singular cohomology group. Let $\mathcal{E}$ be a vector bundle of rank $s$ over $X_{\Sigma}$. The \textit{Chern classes} $c_i(\mathcal{E}) \in A^{i}(X_\Sigma)$ are invariants satisfying the following properties:
    \begin{enumerate}[label = (\roman*)]
        \item[-] $c_0(\mathcal{E}) = 1$ for all vector bundles $\mathcal{E}$ and $c_1(\mathcal{O}) = 0$.
        \item[-] $c_i(\mathcal{E}^{\vee}) = (-1)^i c_i(\mathcal{E})$, where $\mathcal{E}^{\vee}$ is the dual vector bundle.
        \item[-] $c_1(\wedge^s\mathcal{E}) = c_1(\mathcal{E})$.
    \end{enumerate}
    The \textit{total Chern class} $c(\mathcal{E}) \in A^{\bullet}(X)$ equals the polynomial $$c(\mathcal{E}) \coloneqq 1 + c_1(\mathcal{E}) + \dots + c_n(\mathcal{E})$$
    and satisfies the formula
    \begin{equation}
    \label{eq: whitney}
    c(\mathcal{E}) = c(\mathcal{F})c(\mathcal{G})
    \end{equation} for every short exact sequence $0 \xrightarrow[]{} \mathcal{F} \xrightarrow[]{} \mathcal{E} \xrightarrow[]{} \mathcal{G} \xrightarrow[]{} 0$. The \textit{Chern roots} of $\mathcal{E}$ are symbols $x_1,\dots,x_s$ satisfying that
    $$c(\mathcal{E}) = \prod_{i = 1}^s(1 + x_i).$$
    In practical terms, this means that
    $$c_i(\mathcal{E}) = \sigma_i(x_1,\dots,x_s)$$
    where $\sigma_i$ is the $i$-th symmetric polynomial. Each Chern class $c_i(\mathcal{E}) \in A^i(X_\Sigma)$ acts on Chow cohomology by multiplication:
\[
c_i(\mathcal{E}) \cdot (-) : A^j(X_\Sigma) \to A^{j+i}(X_\Sigma).
\]
In particular, for a complete toric variety $X_\Sigma$ of dimension $n$, the top Chern class $c_n(\mathcal{E})$ satisfies
\[
\int_{X_\Sigma} c_n(\mathcal{E}) := \deg(c_n(\mathcal{E}) \cap [X_\Sigma]) \in \mathbb{Z},
\]
where $[X_\Sigma] \in A^0(X_\Sigma)$ is the fundamental class and $A^n(X_\Sigma) \cong \mathbb{Z}$.
\end{definition}

In the case where $\mathcal{E} \cong \oplus_{i = 1}^n\mathcal{O}(D_i)$ for some nef line bundles $D_1,\dots,D_n$, the $n$-th Chern class has a polyhedral interpretation as it coincides with the mixed volume of the corresponding polytopes.

\begin{theorem}
\label[theorem]{thm: MixedVolumeChernClass}
Let $X_{\Sigma}$ be a projective simplicial toric variety and let $D_1,\dots,D_n$ be nef Cartier $\Tn$-invariant divisors on $X_{\Sigma}$, with associated polytopes $\Delta_1,\dots,\Delta_n$.
If $\mathcal{E} = \oplus_{i = 1}^n \mathcal{O}(D_i)$, then
$$\int_{X_{\Sigma}} c_n(\mathcal{E}) = \MV(\Delta_1,\dots,\Delta_n).$$
\end{theorem}

In order to relate the degree of the ideal generated by global sections of $\mathcal{E}$ and their Chern classes, we will use the \textit{Hirzebruch-Riemann-Roch} theorem, which relates the following three objects.

\begin{enumerate}[label = (\roman*)]
    \item The \textit{Chern character} $\ch(\mathcal{E}) \in A^{\bullet}(X_{\Sigma})$ of $\mathcal{E}$ is the power series which, in terms of the Chern roots of $\mathcal{E}$, equals
    $$\ch(\mathcal{E}) = \sum_{i = 1}^se^{x_i} \in A^{\bullet}(X_\Sigma).$$
    \item If $X_{\Sigma}$ is a smooth variety, the \textit{Todd class} $\Td(X_{\Sigma}) \in A^{\bullet}(X_{\Sigma})$ can be defined using the Chern roots $\xi_1,\dots,\xi_n$ of its tangent bundle $\mathcal{T}$ as
    $$\Td(X) \coloneqq \prod_{i = 1}^n \frac{\xi_i}{1 - e^{-\xi_i}}.$$
    If $X_{\Sigma}$ is not smooth, then $\Td(X_{\Sigma})$ can be defined by pushing forward the Todd class of a desingularization (see \cite{Fulton1984}).
    \item The Euler characteristic of a vector bundle $\mathcal{E}$ over $X_{\Sigma}$ equals
    $$\chi(\mathcal{E}) = \sum_{i \geq 0}(-1)^i\dim H^i(X_{\Sigma},\mathcal{E}).$$
\end{enumerate}

\begin{theorem}
\label[theorem]{thm: HRR}
Let $X_{\Sigma}$ be a complete toric variety and let $\mathcal{E}$ be a vector bundle. Then
\begin{equation}
\label{eq: eulerclass}
\chi(\mathcal{E}) = \int_{X_{\Sigma}} \ch(\mathcal{E})\cdot\Td(X).
\end{equation}
\end{theorem}

\section{Vertical systems and toric vector bundles}
\label{section: vertical}

In this section, we describe vertically parametrized polynomial systems and we link them to toric vector bundles. These are families of parametric polynomial systems whose appearance in several applications has attracted interest in recent years. Some of their geometric properties, including the generic root counts over $\Tn$ and the dimension for generic values of the parameters, have been studied in \cite{genericrootcounts, FELIU2025630, tropicalrootbounds, kaveh2025vectorvaluedlaurentpolynomialequations}. To each of these systems, we show how to attach a toric variety $X_{\Sigma}$ and a toric vector bundle $\mathcal{E}_F$, similarly to what has been done in \cite{kaveh2025vectorvaluedlaurentpolynomialequations}.

\subsection{Vertically parametrized systems and vertical pairs}
For the sake of simplicity in the exposition, we work over the lattice $M = \mathbb{Z}^n$ and $M_{\R} \cong \mathbb{R}^n$ and the vector space $E \cong \mathbb{C}^s$. Consider the following objects:
\begin{itemize}
\item[-] A matrix $C\in \C^{s\times m}$ of rank $s$ with columns $\gamma_1,\dots,\gamma_m \in \C^s$ with $s \geq 2$,
\item[-] A matrix $B \in \mathbb{Z}^{n \times m}$ with columns $b_1,\dots,b_m \in \mathbb{Z}^n$.
\end{itemize}
A \emph{vertically parametrized system} (or \emph{vertical system} for short) is a  parametric polynomial system of the form
\begin{equation}
\label{eq:NkB}
F=C(u \star x^B)\in \C[u_1,\ldots,u_m,x_1^\pm,\ldots,x_n^\pm]^s,
\end{equation}
consisting of $s$ linearly independent polynomials with parameters $u=(u_1,\ldots,u_m)$ and variables
$x=(x_1,\ldots,x_n)$,
The component\-wise product $u \star x^B$ indicates that the monomial $x^{b_i}$ corresponding to the $i$-th column of $B$ is scaled by $u_i$, while the rows of $C$ give  linear combinations of the scaled monomials. An important feature is that $F$ is linear in the parameters and that each parameter always accompanies the same monomial. For a vertical system $F$, we denote the specialization at $u \in \C^m$ by
\[F_{u}=F(u,\cdot) \in\C[x^\pm]^{s}\, . \]
Therefore, we consider the set of zeros over $\Tn$ as:
\begin{equation*}
V_{\Tn}(F_{u})=\{x \in \Tn : F_{u}(x)=0 \}\,.
\end{equation*}
Vertical systems have recently appeared in the literature as vector-valued polynomial equations \cite{kaveh2025vectorvaluedlaurentpolynomialequations}. Concretely, for each $u \in \C^m$ one can consider $F_{u}$ as the tensor in $\mathbb{C}^s \otimes \mathbb{C}[x^{\pm}]$ given by
\begin{equation}
\label{eq: Futensor}
F_{u} = \sum_{i = 1}^m u_i (\gamma_i \otimes x^{b_i}) \subset \mathbb{C}^s \otimes \mathbb{C}[x^{\pm}].
\end{equation}
For each $\bar{x} \in \Tn$, the value $F_u(\bar{x})$ is a vector in $\mathbb{C}^s$. The set $V_{\Tn}(F_u)$ consists of the points where $F_u(\bar{x}) = \mathbf{0} \in \mathbb{C}^s$.

\begin{remark}
\label[remark]{remark: reparametrizing}
As illustrated in \Cref{fig: exampleIntro}, the information of a vertical system can be encoded as a set of pairs given by lattice points in $\mathbb{Z}^n$ and the vector spaces spanned by the columns of $C$ that correspond to that lattice point. By the multilinearity of the tensors in \eqref{eq: Futensor}, any basis of that vector space can represent the same vertical system, up to a linear transformation in the parameters. In particular, if $\gamma \in \mathbb{C}^s$ belongs to the vector space spanned by the columns of $C$ that correspond to a fixed lattice point, we can always assume that $\gamma$ is one of the columns of $C$.

For instance, the vertical system in \Cref{first_system} can also be written as
$$\begin{cases}
\overline{u}_1 +  \hspace{0.9cm} \overline{u}_3 x + \overline{u}_4 y + \overline{u}_5 x y \hspace{1.2cm}= 0,\\
 \hspace{0.9cm}   \overline{u}_2 + \overline{u}_3 x  +  \hspace{1.1cm} \overline{u}_5 x y + \overline{u}_6 y = 0.
\end{cases}$$
where the new parameters are $(\overline{u}_1,\overline{u}_2,\overline{u}_3,\overline{u}_4,\overline{u}_5,\overline{u}_6) = (u_1,u_2,u_3,u_4 + u_6,u_5,2u_4 + u_6)$.
\end{remark}

Consider $\Delta_F \subset M_{\R}$ to be the polytope:
\begin{equation}
\label{eq: polytopeMain}
\Delta_F \coloneqq \conv \Big\{\sum_{i \in \mathcal{I}}b_i \in \mathbb{Z}^n : \:  \mathcal{I} \subset [m], \: |\mathcal{I}| = s, \: \bigwedge_{i \in \mathcal{I}}\gamma_{i} \neq 0\Big\} \subset \R^n
\end{equation}
and let $\Sigma_F$ be the normal fan of $\Delta_F$.

\begin{remark}
\label[remark]{rk: MinkowskiSum}
 The polytope $\Delta_F$ can be computed as the projection, via the map
$$p: \mathbb{R}^m \to \mathbb{R}^n, \quad v \mapsto Bv,$$
of the basis polytope of the representable matroid of $C$ (see \cite{Feichtner2005} for the definition of the basis polytope). Note that if $F$ is a sparse system, $\Delta_F$ equals the Minkowski sum of the Newton polytopes of the individual equations. Hence, the use of $X_{\Sigma}$ generalizes the toric varieties employed in \cite{bender2021toric, dandrea2025sparse, BUSE2024107739}.
\end{remark}

Similarly to \cite[Definition 3.1]{kaveh2025vectorvaluedlaurentpolynomialequations}, we attach to each vertical system $F = C(u \star x^B)$ a pair given by a projective toric variety $X_{\Sigma}$ and toric vector bundle $\mathcal{E}_F$.

\begin{definition}
\label[definition]{def: verticalVectorBundle}
Let $F = C(u \star x^B)$ be a vertical system.  We say that $(X_{\Sigma},\mathcal{E}_F)$ is a \emph{vertical pair} associated to $F$ if $X_{\Sigma}$ is a toric variety given by a fan $\Sigma$ and $\mathcal{E}_F$ is a toric vector bundle over $X_{\Sigma}$ satisfying:
\begin{enumerate}[label = (\roman*)]
    \item $X_{\Sigma}$ is projective, $\Sigma$ refines $\Sigma_F$ (i.e. every cone in $\Sigma$ is contained in a cone of $\Sigma_F$) and $\Sigma$ contains a smooth maximal cone $\sigma \in \Sigma(n)$.

    \item The Klyachko filtration associated to $\mathcal{E}_F$ equals
        \begin{equation}
\label{eq: klyachko_data_vps}
    E^{i}_{\rho} \coloneqq \spann \langle \gamma_j : \:   \langle v_{\rho}, b_j \rangle \leq -i \rangle \subset \mathbb{C}^s, \: \rho \in \Sigma(1) \text{ and } i \in \mathbb{Z},
\end{equation}
where $v_{\rho} \in \mathbb{Z}^n$ is the primitive generator of the ray $\rho \in \Sigma(1)$.
\end{enumerate}
We say that $(X_{\Sigma},\mathcal{E}_F)$ is a \emph{simplicial/smooth vertical pair} if $X_{\Sigma}$ is simplicial/smooth.
\end{definition}
It is worth noting that smooth (and thus simplicial) vertical pairs always exist. Indeed, by the toric resolution of singularities \cite[Theorem 11.1.9]{coxlittleschneck}, there exists a smooth refinement $\Sigma$ of $\Sigma_F$ such that the induced toric morphism
\[
X_{\Sigma}\longrightarrow X_{\Sigma_F}
\]
is projective. Since $\Sigma_F$ is the normal fan of the polytope $\Delta_F$, the toric variety $X_{\Sigma_F}$ is projective, and hence so is $X_{\Sigma}$. The following technical lemma shows that the Klyachko filtration from \eqref{eq: klyachko_data_vps} satisfies the compatibility conditions from \Cref{def: KlyachkoFiltration}.

\begin{lemma}
\label[lemma]{lemm: MinkowskiReason}
Let $F = C(u \star x^B)$ be a vertical system and let $(X_{\Sigma},\mathcal{E}_F)$ be a vertical pair associated to $F$ with Klyachko filtration $\EFilt$. For each cone $\sigma \in \Sigma(n)$, let $\mathcal{I} \subset [m]$ be such that $\sum_{i \in \mathcal{I}}b_i$ equals the vertex of $\Delta_F$ associated to $\sigma$.

\begin{enumerate}[label = (\roman*)]
    \item For each $i \in \mathcal{I}$ and $\rho \in \sigma(1)$, we have
\begin{equation}
\label{eq: vertex}
    \langle v_{\rho},b_i \rangle + \max\{i \in \mathbb{Z} : \: \gamma_i \in E_{\rho}^{i}\} = 0.
\end{equation}

\item For all $\rho \in \sigma(1)$ and $i \in \mathbb{Z}$, $E_{\rho}^{i}$ is spanned by a subset of $\{\gamma_i : \: i \in \mathcal{I}\}$. In particular, $\EFilt$ satisfies the compatibility conditions in \eqref{eq: compatibility}.

\end{enumerate}

\end{lemma}

\begin{proof}

We fix a cone $\sigma \in \Sigma(n)$ and consider the vertex of $\Delta_F$ corresponding to $\sigma$. Up to reordering the columns, we may assume that this vertex equals $\sum_{i = 1}^sb_i$, i.e. it is the vertex corresponding to the sum of the first $s$ columns of $B$. This implies that for every other $b \in \Delta_F$, we have
\begin{equation}
\label{eq: vertexeq}
    \langle v_{\rho},b \rangle \geq \langle v_{\rho}, \sum_{i = 1}^s b_i\rangle, \quad \rho \in \sigma(1).
\end{equation}
With this, we define the family of polyhedra
\begin{equation}
\label{eq: polyhedra}
P_i^{\sigma} \coloneqq \{b \in \mathbb{Z}^n : \: \langle v_{\rho},b \rangle \geq \langle v_{\rho},b_i \rangle, \: \forall \rho \in \sigma(1) \}, \quad i \in \{1,\dots,s\}.
\end{equation}
We claim that if $b_j \notin P_i^{\sigma}$ for some $j \in [m]$, then $\gamma_j \in \spann\langle \gamma_1,\dots,\gamma_{i-1},\gamma_{i+1},\dots,\gamma_s \rangle$. Namely, if $b_j \notin P_i^{\sigma}$, then there exists $\rho \in \sigma(1)$ such that
$$\langle v_{\rho},b_j \rangle < \langle v_{\rho},b_i \rangle. $$
If $\gamma_j \notin \spann\langle \gamma_1,\dots,\gamma_{i-1},\gamma_{i+1},\dots,\gamma_s \rangle$, then
$b_j + \sum_{k \neq i}b_k \in \Delta_F$ and so $$\langle v_{\rho},b_j + \sum_{k \neq i}b_k \rangle < \langle v_{\rho}, \sum_{k = 1}^s b_k \rangle $$
contradicting \eqref{eq: vertexeq}. As a consequence,
$$\max (i \in \mathbb{Z} : \: \gamma_j \in E_{\rho}^i) = -\langle v_{\rho},b_j \rangle, $$
for each $\rho \in \sigma(1)$, implying \eqref{eq: vertex}.

To prove part (ii), we may assume that $\langle v_{\rho},b_1 \rangle \leq \dots \leq \langle v_{\rho},b_s \rangle$. For this purpose, we must show that for all $\ell = 1,\dots,s$
\begin{equation}
\label{eq: maxinequality}\max(i \in \mathbb{Z} : \: \dim E_{\rho}^i \geq \ell) = \max(i \in \mathbb{Z} : \: \gamma_1,\dots,\gamma_{\ell} \in E_{\rho}^i)
\end{equation}
Clearly, $\max(i \in \mathbb{Z} : \: \dim E_{\rho}^i \geq \ell) \geq \max(i \in \mathbb{Z} : \: \gamma_1,\dots,\gamma_{\ell} \in E_{\rho}^i)$. If the inequality were strict, then there would exist $j \in [m] \setminus [s]$ such that
$$
\max(i \in \mathbb{Z} : \: \gamma_1,\dots,\gamma_{\ell-1},\gamma_j \in E_{\rho}^i) > \max(i \in \mathbb{Z} : \: \gamma_1,\dots,\gamma_{\ell} \in E_{\rho}^i),
$$
with $\gamma_j \notin \spann\langle \gamma_1,\dots,\gamma_{\ell-1} \rangle$. In particular, $\gamma_j = \sum_{k = 1}^s\alpha_k\gamma_k$ for some $\alpha_k \in \mathbb{C}$, with $\alpha_k \neq 0$ for some $k \geq \ell$. We pick $k \in \{\ell,\dots,s\}$ such that $\alpha_k \neq 0$ and note that
$\gamma_1 \wedge \dots \wedge \gamma_{k -1} \wedge \gamma_j \wedge \gamma_{k + 1} \wedge \dots \wedge \gamma_{s} \neq 0$ and
$$\max(i \in \mathbb{Z} : \: \gamma_1,\dots,\gamma_{\ell-1},\gamma_j \in E_{\rho}^i) > \max(i \in \mathbb{Z} : \: \gamma_1,\dots,\gamma_{\ell - 1},\gamma_{k} \in E_{\rho}^i)$$
However, this would imply that
$$\langle v_{\rho},b_j \rangle + \langle v_{\rho}, \sum_{t \neq k}b_t\rangle  < \langle v_{\rho},\sum_{t = 1}^{\ell}b_t\rangle,$$
contradicting \eqref{eq: vertexeq}. Therefore, \eqref{eq: maxinequality} implies that each of the vector spaces $E_{\rho}^i$ for $\rho \in \sigma(1)$ is generated by a subset of $\{\gamma_i : \: i = 1,\dots,s \}$, implying \eqref{eq: compatibility}.
\end{proof}

The results in the previous lemma imply that $\mathcal{E}_F$ is indeed a toric vector bundle which decomposes over $U_{\sigma}$ as:
\begin{equation}
\label{eq: trivialization}
\mathcal{E}_F{\mid _{U_{\sigma}}} \cong \bigoplus_{i \in \mathcal{I}} \mathcal{L}_{b_i}{\mid _{U_{\sigma}}},
\end{equation}
i.e. the weights in the decomposition correspond to the lattice points that give rise to the vertex corresponding to $\sigma$ in $\Delta_F$.

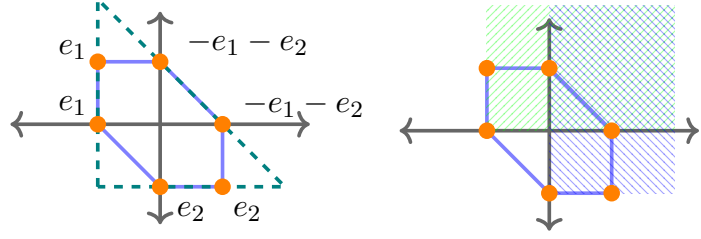
\begin{figure}
    \centering
\scalebox{1.1}{ \begin{tikzpicture}[x=0.75cm, y=0.75cm, line width=1.25pt]

      \foreach \x in {-1,1} {
        \draw[color=white!40!black] (\x, 2pt) -- (\x, -2pt);}
      \foreach \y in {-1,1} {
        \draw[color=white!40!black] (2pt, \y) -- (-2pt, \y);}
      \draw[color=white!40!black, <->] (-2.4, 0) -- (2.4,0) {};
      \draw[color=white!40!black, <->] (0, -1.6) -- (0, 1.8) {};

      \draw[color=blue, opacity=0.5] (1,0) -- (0,1) -- (-1,1) -- (-1,0) -- (0,-1) -- (1,-1) --(1,0) -- cycle;

       \draw[color=teal, dashed] (-1,-1) -- (-1,2) -- (2,-1) -- (-1,-1) -- cycle;

      \node () at (2.3,0.3) {\small $-e_1 - e_2$};
      \node () at (1.4,1.3)  {\small $-e_1 - e_2$};
      \node () at (-1.4,1.2) {\small $e_1$};
      \node () at (-1.4,0.3)  {\small $e_1$};
      \node () at (0.5,-1.4) {\small $e_2$};
      \node () at (1.4,-1.4)  {\small $e_2$};

     \node[circle, fill=orange, inner sep=2.0pt] () at (0,-1) {};
      \node[circle, fill=orange, inner sep=2.0pt] () at (1,-1) {};
      \node[circle, fill=orange, inner sep=2.0pt] () at (0,1) {};
      \node[circle, fill=orange, inner sep=2.0pt] () at (1,0) {};
       \node[circle, fill=orange, inner sep=2.0pt] () at (-1,1) {};
      \node[circle, fill=orange, inner sep=2.0pt] () at (-1,0) {};
    \end{tikzpicture}}
    \centering
\scalebox{1.1}{ \begin{tikzpicture}[x=0.75cm, y=0.75cm, line width=1.25pt]

      \foreach \x in {-1,1} {
        \draw[color=white!40!black] (\x, 2pt) -- (\x, -2pt);}
      \foreach \y in {-1,1} {
        \draw[color=white!40!black] (2pt, \y) -- (-2pt, \y);}
      \draw[color=white!40!black, <->] (-2.4, 0) -- (2.4,0) {};
      \draw[color=white!40!black, <->] (0, -1.6) -- (0, 1.8) {};

      \draw[color=blue, opacity=0.5] (1,0) -- (0,1) -- (-1,1) -- (-1,0) -- (0,-1) -- (1,-1) --(1,0) -- cycle;

\usetikzlibrary{patterns}
\fill[pattern=north east lines, pattern color=green, opacity=0.6] (-1,0) rectangle (2,2);
\fill[pattern=north west lines, pattern color=blue, opacity=0.6] (0,-1) rectangle (2,2);

     \node[circle, fill=orange, inner sep=2.0pt] () at (0,-1) {};
      \node[circle, fill=orange, inner sep=2.0pt] () at (1,-1) {};
      \node[circle, fill=orange, inner sep=2.0pt] () at (0,1) {};
      \node[circle, fill=orange, inner sep=2.0pt] () at (1,0) {};
       \node[circle, fill=orange, inner sep=2.0pt] () at (-1,1) {};
      \node[circle, fill=orange, inner sep=2.0pt] () at (-1,0) {};
    \end{tikzpicture}}
    \caption{An illustration of \Cref{lemm: MinkowskiReason} for the vertical system in \Cref{ex: TangentExample}. In the first image, we see the vertical system with the polytope $\Delta_F$ drawn in green. In the second, we draw the polyhedra $P_i^{\sigma}$ in \eqref{eq: polyhedra} corresponding to the vertex $(-1,-1)$ of $\Delta_F$. Outside the \textcolor{green}{green} (resp. \textcolor{blue}{blue}) dashed polyhedra, all vectors $\gamma_i$ are contained in $\langle e_2 \rangle$ (resp. $\langle e_1 \rangle$).}
    \label{fig: exampleIntro2}
\end{figure}

\begin{example}
\label[example]{ex: exampleContinued2}
We consider the vertically parametrized system $F = C(u \star x^B) \subset \mathbb{C}^2 \otimes \mathbb{C}[x_1^{\pm},x_2^{\pm}]$ from \eqref{first_system}, which has matrices
\[
C = \begin{pmatrix}
    1 & 0 & 1 & 1 & 1 & 1 \\
    0 & 1 & 1 & 2 & 1 & 1
\end{pmatrix}, \qquad
B = \begin{pmatrix}
    0 & 0 & 1 & 0 & 1 & 0 \\
    0 & 0 & 0 & 1 & 1 & 1
\end{pmatrix}.
\]
A straightforward computation reveals that
$$\Delta_F = \conv \{(0,0),(0,2),(1,0),(1,2)\}.$$
The normal fan $\Sigma \subset N_{\R}$ of $\Delta_F$ has rays $\Sigma(1) = \{\rho_1,\rho_2,\rho_3,\rho_4\}$ with primitive generators, $\{(1,0),(0,1),(-1,0),(0,-1)\}$, respectively. The underlying toric variety is $X_{\Sigma} = \mathbb{P}^1 \times \mathbb{P}^1$, which is smooth and projective. The toric vector bundle $\mathcal{E}_F$ associated to this system is given by the Klyachko filtration $(E_{\rho}^{\bullet})_{\rho \in \Sigma(1)}$ with terms
\begin{center}
\begin{tabular}{c|c|c|c}
$\Sigma(1)$  & $E_{\rho}^{0}$ & $E_{\rho}^{1}$ & $E_{\rho}^{2}$ \\ \hline
$(1,0)$  & $\mathbb{C}^2$ & $0$ & $0$ \\
$(0,1)$  & $\mathbb{C}^2$ & $0$ & $0$\\
$(-1,0)$   & $\mathbb{C}^2$ & $\mathbb{C}^2$ & $0$\\
$(0,-1)$  & $\mathbb{C}^2$ & $\langle(1,1)\rangle$ & $0$
\end{tabular}.
\end{center}
One can verify that the toric vector bundle $\mathcal{E}_F$ can be split as $$\mathcal{E}_F \cong \mathcal{O}(1,1) \oplus \mathcal{O}(0,1),$$ by showing that the filtrations in $\EFilt$ can be written as a direct sum of filtrations over a $1$-dimensional vector space.
\end{example}

\begin{example}
\label[example]{ex: TangentExample}
    Consider the polynomial system $F = C(u \star x^B)$ given by matrices:
    $$C = \begin{pmatrix}
        1 & 1 & 0 & 0 & -1 & -1 \\
        0 & 0 & 1 & 1 & -1 & -1
    \end{pmatrix}, \quad B = \begin{pmatrix}
        -1 & -1 & 1 & 0 & 1 & 0 \\
        1 & 0 & -1 & -1 & 0 & 1
    \end{pmatrix},$$
    with parameters $u_1,\dots,u_6$. This corresponds to the system of Laurent polynomial equations
\begin{equation}
\label{second_system}
\begin{cases}
u_1 x^{-1} y + u_2 x^{-1}  \hspace{3.1cm}- u_5 x - u_6y= 0,\\[4pt]
\hspace{3.1cm} u_3xy^{-1} + u_4 y^{-1} - u_5 x - u_6y = 0.
\end{cases}
\end{equation}
The polytope $\Delta_F$ equals $\conv \{(-1,-1),(-1,2),(2,-1)\}$. The toric variety $X_{\Sigma}$ given by $\Sigma_F$ equals $\mathbb{P}^2$, which is smooth and projective. In particular, $\Sigma(1) = \{\rho_1,\rho_2,\rho_3\}$, where the primitive generators of the rays are $$ \{(-1,-1),(1,0),(0,1)\} \subset \mathbb{Z}^2.$$ The Klyachko filtration providing the toric vector bundle $\mathcal{E}_F$ associated to this system equals
\begin{equation}
\label{ex: KlyachkoOfTangent}
\begin{tabular}{c|c|c|c}
$\Sigma(1)$ & $E_{\rho}^0$ & $E_{\rho}^1$& $E_{\rho}^2$\\ \hline

$(-1,-1)$ & $\mathbb{C}^2$ & $\langle -e_1 - e_2\rangle$ & $0$ \\

$(1,0)$  & $\mathbb{C}^2$ & $\langle e_1 \rangle$ & $0$ \\

$(0,1)$ & $\mathbb{C}^2$ & $\langle e_2\rangle$ & $0$
\end{tabular}.
\end{equation}

 Note that this toric vector bundle does not split as a direct sum of line bundles. In fact, $\mathcal{E}_F$ equals the tangent bundle over $\mathbb{P}^2$ (see \cite[Example 3.8]{parliament}). In particular, we can see that for all linear transformations of the system \eqref{first_system}, there exists a face of their Newton polytopes which is not overdetermined.

\end{example}

\subsection{Homogenization}

In order to define homogenization of a vertical system that will be used to attach a toric vector bundle, we first need to recall the toric way to homogenize a polynomial system, which has been used previously in the context of studying parametric polynomial systems \cite{bender2021toric, dandrea2025sparse}.

Throughout this section $R$ denotes the Cox ring of the variety $X_{\Sigma}$ and $R^{\pm} \coloneqq \mathbb{C}[z_{\rho}^{\pm} : \: \rho \in \Sigma(1)]$ denotes the ring of Laurent polynomials in the variables of the Cox ring. Similarly to \eqref{eq: CoxRing}, we have
$$R^{\pm} = \bigoplus_{\alpha \in \Cl(X_{\Sigma})}R^{\pm}_{\alpha} \cong \bigoplus_{\alpha \in \Cl(X_{\Sigma})}H^0(\Tn,\mathcal{O}(D))$$
where $\alpha = [D] \in \Cl(X_{\Sigma})$. Moreover, we will denote by $A = \mathbb{C}[u_1,\dots,u_m]$ the ring of coefficients. The following map provides a method to homogenize Laurent polynomials in $\C[x^{\pm}]$ to polynomials in $R^{\pm}$.
\begin{lemma}
\label[lemma]{lemma: homogenize}
Let $X_{\Sigma}$ be a complete toric variety. For each $D = \sum_{\rho \in \Sigma(1)}a_{\rho}D_{\rho} \in \Div(X_{\Sigma})$ with $\alpha = [D] \in \Cl(X_{\Sigma})$, the following hold:
\begin{enumerate}[label = (\roman*)]
    \item $z^{\mathbf{V}b + a} \in R^{\pm}_{\alpha}$  and
    \item $z^{\mathbf{V}b + a} \in R_{\alpha}$, if and only if $b \in \Delta_D$.
\end{enumerate}
\end{lemma}

\begin{proof}
As the Gale dual $\pi$ of $\mathbf{V}$ defines the grading in the ring $R_{\alpha}^{\pm}$, we have that $\deg(\mathbf{V}b + a) = [D] \in \Cl(X_{\Sigma})$. Moreover, in order for $z^{\mathbf{V}b + a}$ to be globally defined, we need to verify that $z^{\mathbf{V}b + a} \in R_{\alpha}$. As the rows of $\mathbf{V}$ are the primitive generators of the rays in $\Sigma(1)$, this amounts to verifying that
\[\langle v_{\rho},b\rangle + a_{\rho} \geq 0, \quad \forall \rho \in \Sigma(1).
\]
This happens if and only if $b \in \Delta_{D}$, concluding the proof.
\end{proof}

Together with the vector bundle $\mathcal{E}_F$, we also attach to $F$ the family of $\Tn$-invariant Weil divisors $D_{F,1},\dots,D_{F,s} \in \Div(X_{\Sigma})$ defined as
$$D_{F,\ell} \coloneqq
    D_{\mathcal{E}_F,\ell} = \sum_{\rho \in \Sigma(1)}\omega_{\rho,\ell}D_{\rho}, \quad \ell \in \{1,\dots,s\}
$$
i.e. the jump divisors associated to $\mathcal{E}_F$, as in \eqref{eq: jump}. We denote by $\omega_{\ell} \in \mathbb{Z}^{\Sigma(1)}$ the vector of coefficients of $D_{F,\ell}$, and by $\Delta_{F,\ell}$ the polytope associated with $D_{F,\ell} $, as in \eqref{eq:polytope}, for $\ell = 1,\dots,s$. We also denote
\begin{equation}
\label{lambdaj}
\lambda_{\ell} = \sum_{j = 1}^{\ell}\omega_{s - j + 1}.
\end{equation}

\begin{definition}
\label[definition]{def: sectionhomgenization}
Let $F = C(u \star x^B)$ be a vertical system and let $(X_{\Sigma},\mathcal{E}_F)$ be a vertical pair associated to it. For $\ell \in \{1,\dots,s\}$, we define the tensor
\begin{equation}
\label{eq: widetilde}\widetilde{F}^{\ell} \coloneqq \sum_{i = 1}^m u_i  (\gamma_i \otimes z^{\mathbf{V}b_i + \omega_{\ell}}) \in A \otimes  \mathbb{C}^s  \otimes R^{\pm}.
\end{equation}
In particular, we denote $\widetilde{F} \coloneqq \widetilde{F}^{s}$.
\end{definition}

For each $u \in \C^m$ and $\ell \in \{1,\dots,s\}$, the specialization $\widetilde{F}^{\ell}_{u} \in \mathbb{C}^s  \otimes R^{\pm}$ describes a tuple of Laurent polynomials of degree $[D_{F,\ell}] \in \Cl(X_{\Sigma})$. In particular, as $\widetilde{F}^{\ell}_{u}$ is a tuple of rational functions whose denominators do not vanish on $\Tn$, we can consider its zero locus $V_{\Tn}(\widetilde{F}^{\ell}_{u})$.

\begin{lemma}
\label[lemma]{lemmaSectionsTorus}
Let $F = C(u \star x^B)$ be a vertical system and let $(X_{\Sigma},\mathcal{E}_F)$ be a vertical pair associated to it. For all $u \in \C^m$ and $\ell \in \{1,\dots,s\}$, there exists an isomorphism of schemes $$V_{\Tn}(\widetilde{F}^{\ell}_{u}) \cong V_{\Tn}(F_{u}).$$
\end{lemma}

\begin{proof}
Let $\sigma \in \Sigma(n)$ be a smooth maximal cone with rays $\{\rho_1,\dots,\rho_n\}$ and generators $\{v_1,\dots,v_n\}$, which form a $\mathbb{Z}$-basis of $\mathbb{Z}^n$. Therefore, the map
\begin{equation}
\label{eq: theisomorphism}
\Tn \xrightarrow{} \Tn ,\quad x_i \xrightarrow[]{} x^{v_{i}} \coloneqq \prod_{j = 1}^n x_j^{v_{i,j}}.
\end{equation}
is an automorphism of $\Tn$. For each $\overline{x} \in \Tn$, we can define $\overline{z} \in  X_{\Sigma}$ (in the coordinates of the Cox ring) as
$$\overline{z}_{\rho} \coloneqq \begin{cases}
    x_i & \rho = \rho_i \\
    1 & \text{otherwise.}
\end{cases}.$$
Since $U_{\sigma}$ is smooth, the restriction of every Weil divisor to $U_{\sigma}$ is Cartier \cite[Proposition 4.2.6]{coxlittleschneck}. Therefore, we may assume that there exists $m_{\sigma} \in \mathbb{Z}^n$ satisfying \eqref{eq: cartierData} for $D_{F,\ell}$.
Up to multiplying $\widetilde{F}^{\ell}$ by a monomial (which does not change $V_{\Tn}(\widetilde{F}^{\ell})$), we may assume that $m_{\sigma} = 0$, and so  $\omega_{\rho,\ell} = 0$ for all $\rho \in \sigma(1)$. Therefore,
$$\widetilde{F}_u^{\ell}(\overline{z}) = \sum_{i = 1}^m u_i(\gamma_i \otimes \overline{z}^{\mathbf{V}b_i + \omega_{\ell}}) = \sum_{i = 1}^{m}u_i\big(\gamma_i \otimes \prod_{j = 1}^n \overline{x}_j^{\langle v_{j},b_i \rangle}\big)$$
Using the isomorphism \eqref{eq: theisomorphism}, we see that $F_u(x) = 0$, if and only if $\widetilde{F}_u^{\ell}(z) = 0$.
\end{proof}

\subsection{Vertical systems and homogeneous ideals}

Next, we construct a homogeneous ideal $I(F) \subset R$ in the Cox ring of the toric variety $X_{\Sigma}$ which has the same zeros in $\Tn$ as the vertically parametrized system $F = C(u \star x^B)$, i.e. $V_{\Tn}(F) = V_{\Tn}(I(F))$. In addition, we show that for generic values of $u \in \mathbb{C}^m$, $V(I(F_u))$ has no solutions at infinity.

\begin{definition}
\label[definition]{def: HE}
For a toric vector bundle $\mathcal{E}$, with Klyachko filtration $\EFilt$, we define the vector space
$$H_{\mathcal{E}} \coloneqq \langle e \otimes z^{\mathbf{V}b + \omega_1} : \: b \in \mathbb{Z}^n, \: e \in V^b\rangle \subset \mathbb{C}^s \otimes R_{\alpha}^{\pm}$$
where $\alpha = [D_{\mathcal{E},1}]$ and $D_{\mathcal{E},1} = \sum_{\rho \in \Sigma(1)}\omega_{\rho,1}D_{\rho} \in \Div(X_{\Sigma})$.
\end{definition}

\begin{lemma}
\label[lemma]{thm: globalSections}
Let $\mathcal{E}$ be a toric vector bundle over a toric variety $X_{\Sigma}$ with Klyachko filtration $\EFilt$. Then,
\begin{enumerate}[label = (\roman*)]
    \item If $b \notin \Delta_{D_{\mathcal{E},1}}$, then $V^b = 0$.
    \item $H_{\mathcal{E}} \subset \mathbb{C}^s \otimes R_{\alpha}$. In particular, if $e \otimes z^{\mathbf{V}b + \omega_1} \in H_{\mathcal{E}}$, then $b \in \Delta_{\mathcal{E}}(e)$.
    \item $H^0(X_{\Sigma},\mathcal{E}) \cong H_{\mathcal{E}}$.
\end{enumerate}

\end{lemma}

\begin{proof}
If $b \notin \Delta_{D_{\mathcal{E},1}}$, then there exists $\rho \in \Sigma(1)$ such that $-\langle v_{\rho},b \rangle > \max\{i \in \mathbb{Z} : \: E_{\rho}^i \neq 0\}$. This implies that $E_{\rho}^{-\langle v_{\rho},b \rangle} = 0$, and so $V^b = 0$, proving part (i).  Moreover, if $e \in V^{b}$, then
$e \in E_{\rho}^{-\langle v_{\rho},b \rangle}$ for all $\rho \in \Sigma(1)$, which implies that $$\langle v_{\rho},b \rangle + \omega_{\rho,1} \geq \langle v_{\rho},b \rangle +\max (i \in \mathbb{Z} : \: e \in E^{i}_{\rho}) \geq 0,$$ deducing (ii). By the isomorphism in \Cref{thm: globalSectionsVu}, (iii) follows from homogenizing the monomials $e \otimes x^b \in \mathbb{C}^s \otimes \C[x^{\pm}]$.
\end{proof}

\begin{remark}
\label[remark]{rk: nonGlobalSections}
Note that the isomorphism in \cref{thm: globalSections} implies that describing a map between the global sections of two toric vector bundles $\mathcal{E},\mathcal{E}'$ amounts to describing a map between $H_{\mathcal{E}}$ and $H_{\mathcal{E}'}$. Moreover, sections of a vector bundle $\mathcal{E}$ over an open subset $U$ can be (locally) written as $f/g$ where $f \in H^0(X_{\Sigma},\mathcal{E} \otimes \mathcal{O}(D))$ and $g \in H^0(X_{\Sigma}, \mathcal{O}(D))$ for some $D \in \Div(X_{\Sigma})$ where $g$ does not vanish on $U$. Therefore, \Cref{thm: globalSections} can also be used to describe local sections of $\mathcal{E}$.
\end{remark}

Note that if $\f \in (\mathbb{C}^s)^{\vee}$, we may consider the polynomial
$$\f(\widetilde{F}^{\ell}) = \sum_{j = 1}^m\f(\gamma_j)u_jz^{\mathbf{V}b_j + \omega_{\ell}} \in A \otimes R^{\pm}.$$
The construction of the ideal $\IFk \subset R$ that homogenizes the elements in the vertical family relies on the following lemma.

\begin{lemma}
\label[lemma]{lemma: mainHomogenization}
      Let $F = C(u \star x^B)$ be a vertically parametrized system, let $(X_{\Sigma},\mathcal{E}_F)$ be a vertical pair associated to $F$ and let $D = \sum_{\rho \in \Sigma(1)}a_{\rho}D_{\rho} \in \Div(X_{\Sigma})$.  For each $\ell \in \{1,\dots,s\}$, we consider the toric vector bundle
$$ \mathcal{E}^{\ell}(D) \coloneqq \wedge^{\ell}\mathcal{E}^{\vee}_F \otimes \mathcal{O}(D).$$

Then, for each $(\f_1 \wedge \dots \wedge \f_{\ell}) \otimes z^{m} \in H_{\mathcal{E}^{\ell}(D)}$ and $j \in \{1,\dots,\ell\}$, we have
      $$(\varphi_1 \wedge \dots \wedge \widehat{\f_j} \wedge \dots \wedge \f_{\ell}) \otimes z^m \f_j(\widetilde{F}_u^{s - \ell + 1})\in H_{\mathcal{E}^{\ell - 1}(D)}.$$
    for all $u \in \mathbb{C}^m$.
\end{lemma}

\begin{proof}
Using the antisymmetry of $\varphi_1 \wedge \dots \wedge \f_{\ell} \in \wedge^{\ell}(\mathbb{C}^s)^{\vee}$, it is sufficient to show that if $ (\varphi_1 \wedge \dots \wedge \f_{\ell}) \otimes z^{m} \in H_{\mathcal{E}^{\ell}(D)}$, then
     $$(\f_2 \wedge \dots \wedge \f_{\ell}) \otimes  z^{m + \mathbf{V}b_i + \omega_{s - \ell + 1}} \in H_{\mathcal{E}^{\ell - 1}(D)},$$
     for all $i \in [m]$ such that $\f_1(\gamma_i) \neq 0$. By \Cref{prop: maxminwedgeDual}, we have
$$D_{\mathcal{E}^{\ell}(D),1} = D_{\wedge^{\ell}\mathcal{E}^{\vee}_F,1} + D = \sum_{j = 1}^{\ell}D_{\mathcal{E}^{\vee}_F,j} + D = D - \sum_{j = 1}^{\ell}D_{\mathcal{E}_F,s - j + 1} = D - \sum_{j = s - \ell + 1}^s D_{\mathcal{E}_F,j}$$ implying that $m = \mathbf{V}b  + a  - \lambda_{\ell} \in \mathbb{Z}^{\Sigma(1)}$ for some $b \in \mathbb{Z}^n$. Therefore,
$$m + \mathbf{V}b_i + \omega_{s - \ell + 1} = \mathbf{V}(b + b_i) + a  - \lambda_{\ell - 1}$$
which implies that $z^{m + \mathbf{V}b_i + \omega_{s - \ell + 1}}$ has degree $D_{\mathcal{E}^{\ell - 1}(D)}$. 

Let $\EFilt$ (resp. $\EFiltl$) be the Klyachko filtration of $\mathcal{E}_F$ (resp. $\mathcal{E}^{\ell}(D)$). By \Cref{def: HE}, we have that if $(\varphi_1 \wedge \dots \wedge \f_{\ell}) \otimes z^m \in H_{\mathcal{E}^{\ell}(D)}$, then
\begin{equation}
\label{eq: dual_e}
\varphi_1 \wedge \dots \wedge \f_{\ell} \in \bigcap_{\rho \in \Sigma(1)}E(\ell)^{-\langle v_{\rho}, b \rangle }_{\rho}.
\end{equation}
Using the description of the filtration of $\EFiltl$ in \eqref{eq: wedgeFiltration} and \eqref{eq: dualFiltration}, \eqref{eq: dual_e} implies that for each $\rho \in \Sigma(1)$, there exist $i_1,\dots,i_{\ell} \in \mathbb{Z}$ such that
$$\f_j \in (E^{\vee})_{\rho}^{i_j}, \: \forall j \in [\ell], \quad \sum_{j = 1}^{\ell}i_j \geq -\langle v_{\rho},u \rangle - a_{\rho} , \quad \forall \rho \in \Sigma(1).$$
If $\varphi_1(\gamma_i) \neq 0$ for some $i \in [m]$, then
$$\gamma_i \notin E_{\rho}^{1-i_1} \quad \forall \rho \in \Sigma(1).$$
By the definition of $\EFilt$ in \Cref{def: verticalVectorBundle}, this implies that $\langle v_{\rho}, b_i\rangle >  i_1 - 1$ for all $\rho \in \Sigma(1)$. Therefore,
$$-\langle v_{\rho},b + b_i \rangle - a_{\rho} \leq \sum_{j = 1}^{\ell}i_j -\langle v_{\rho}, b_i \rangle \leq \sum_{j = 2}^{\ell}i_j.$$
This implies that $\f_2 \wedge \dots \wedge \f_{\ell} \in \bigcap_{\rho \in \Sigma(1)}E(\ell - 1)_{\rho}^{-\langle v_{\rho}, b + b_i \rangle }$, and so
$$\f_2\wedge \dots  \wedge  \f_{\ell} \otimes z^{m + \mathbf{V}b_i + \omega_{s - \ell + 1}} \in H_{\mathcal{E}^{\ell - 1}(D)},$$
proving the result.
\end{proof}

In the particular case where $\ell = 1$, \Cref{lemma: mainHomogenization} implies that if $\f \otimes z^m \in H_{\mathcal{E}^{1}(D)}$, then
$$
    z^m \f(\widetilde{F}_u) = \sum_{i = 1}^m \f(\gamma_i)u_i z^{m + \mathbf{V}(b_i) + \omega_s} \in H^0(X_{\Sigma},\mathcal{O}(D)) \cong R_{\alpha},
$$
for $\alpha = [D] \in \Cl(X_{\Sigma})$. This motivates the following definition.

\begin{definition}
\label[definition]{def: idealDF}
    Let $F = C(u \star x^B)$ be a vertical system and let $(X_{\Sigma},\mathcal{E}_F)$ be a simplicial vertical pair associated to $F$. Consider a pair $(D_F,\IF)$ satisfying that
    \begin{enumerate}[label = (\roman*)]
        \item $D_F \in \Div(X_{\Sigma})$ is nef and all divisors in the poset associated to $\mathcal{E}_F^{\vee}(D_F)$ are nef.
        \item $\IF \subset A \otimes R$ is the homogeneous ideal
        \begin{equation}
        \label{def: ideal} I(F) \coloneqq \langle  z^m\f(\widetilde{F}) : \: \f \otimes z^{m} \in H_{\mathcal{E}_F^{\vee}(D_F)} \rangle \subset A \otimes R.
        \end{equation}
    \end{enumerate}
    For each value of the parameters $u \in \mathbb{C}^m$, we denote by $\IFk \subset R$ the homogeneous ideal after specializing the parameters. The existence of nef Cartier divisors $D_F$ such that all divisors in the poset associated to $\mathcal{E}_F^{\vee}(D_F)$ are nef is guaranteed by \Cref{lem: neflyDecoratedExists}.
\end{definition}

For the following theorem, we fix a cone $\sigma \in \Sigma(n)$ and assume that the vertex corresponding to $\sigma$ in $\Delta_F$ equals the sum of the first $s$ columns in $B$, i.e. $\sum_{i = 1}^sb_i$. We let $\gamma_1,\dots,\gamma_s \in \mathbb{C}^s$ be the corresponding columns of $C$, which form a basis and $\f_1,\dots,\f_s \in (\mathbb{C}^s)^{\vee}$ a dual basis, i.e. $\f_i(\gamma_j) = \delta_{ij}$ for all $i,j = 1,\dots,s$.

With this, we show that for all $u \in \mathbb{C}^m$ the ideal $\IFk_{\sigma} \subset R_{\sigma}$ (after localizing in the open subset $U_{\sigma}$) is generated by the polynomials
\begin{equation}
\label{eq: Fisigma}
    F_i^{\sigma} \coloneqq z^{-\mathbf{V}(b_i) - \omega_{s}} \f_i (\widetilde{F}_u) = \sum_{j = 1}^m \f_i(\gamma_j) u_j z^{\mathbf{V}(b_j - b_i)} \in R_{\sigma}, \quad i \in [s]
\end{equation}
Note that these polynomials are of the form
\begin{equation}
\label{equationFG}
F_i^{\sigma} = u_i + \sum_{j \notin [s]}\f_i(\gamma_j)u_jz^{\mathbf{V}(b_j - b_i)}, \quad i \in [s].
\end{equation}

\begin{theorem}
\label[theorem]{thm: Torus}
Let $F = C(u \star x^B)$ be a  vertical system and let $(X_{\Sigma},\mathcal{E}_F)$ be a simplicial vertical pair associated to it.
\begin{enumerate}[label = (\roman*)]
\item For all $u \in \mathbb{C}^m$, the $\Tn$-schemes defined by $\IFk$ and $F_u$ are isomorphic.
\item For each $\sigma \in \Sigma(n)$, the ideal $\IFk_{\sigma} \subset R_{\sigma}$ is generated by the polynomials in \eqref{eq: Fisigma}.
\end{enumerate}
\end{theorem}

\begin{proof}

As $D_F$ is a nef divisor, we let $D_F = \sum_{\rho \in \Sigma(1)}a_{\rho}D_{\rho}$ and $b_{\sigma} \in \mathbb{Z}^n$ satisfying
\begin{equation}
\label{eq: sigmamonomials}
\langle v_{\rho},b_{\sigma} \rangle = -a_{\rho}, \quad \forall \rho \in \sigma(1).
\end{equation}
We also denote as $\EFilt$ the Klyachko filtration of $\mathcal{E}_F$, while $\EFiltw$ is the one of $\mathcal{E}^{\vee}_F(D_F)$. By \Cref{lemm: MinkowskiReason}, the vector spaces $E_{\rho}^{i}$ (respectively, $\widetilde{E}_{\rho}^i$) for $\rho \in \sigma(1)$ are spanned by a subset of $\{\gamma_1,\dots,\gamma_s\}$ (respectively, $\{\f_1,\dots,\f_s\}$).

Firstly, we claim that the monomials $$z^{m_1},\dots,z^{m_s} \in R, \quad m_j = \mathbf{V}(b_{\sigma} - b_i) + a - \omega_s$$ satisfy that $\f_j \otimes z^{m_j} \in H_{\mathcal{E}_F^{\vee}(D_F)}$ for all $j \in \{1,\dots,s\}$. Using \Cref{lemm: MinkowskiReason} and \eqref{eq: sigmamonomials}, we see that for all $\rho \in \sigma(1)$,
\begin{multline}
\label{eq: rhoi}
\langle v_{\rho}, b_{\sigma} - b_j \rangle = -a_{\rho} + \max(i \in \mathbb{Z} : \: \gamma_j \in E_{\rho}^i) = \\ -a_{\rho} - \max(i \in \mathbb{Z} : \: \f_j \in (E_{\rho}^{1-i})^{\perp}) = -\max(i \in \mathbb{Z} : \: \f_j \in \widetilde{E}_{\rho}^i)
\end{multline}
As $X_{\Sigma}$ is simplicial, all Weil divisors are $\mathbb{Q}$-Cartier. The assumption in \cref{def: idealDF} implies that the divisors $D_{\mathcal{E}_F^{\vee}(D_F)}(\f_j)$ are nef for all $j \in \{1,\dots,s\}$. Therefore, \eqref{eq: rhoi} implies that $b_{\sigma} - b_j$ lies in the polytope $\Delta_{\mathcal{E}_F^{\vee}(D_F)}(\f_j)$. In particular,
$$\langle v_{\rho}, b_{\sigma} - b_j \rangle \geq -\max(i \in \mathbb{Z} : \: \f_j \in \widetilde{E}_{\rho}^i), \quad \forall \rho \in \Sigma(1) $$
implying $\f_j \in V^{b_{\sigma} - b_j}$ for all $j \in \{1,\dots,s\}$, proving the claim.

Next, we prove the two parts of the theorem. To prove (i), we consider $z \in \Tn$ and show that $\widetilde{F}(z) = 0$, if and only if $\f \otimes z^m(\widetilde{F}(z)) = 0$ for all $\f \otimes z^m \in H_{\mathcal{E}_F^{\vee}(D_F)}$. If $\widetilde{F}(z) = \mathbf{0} \in \mathbb{C}^s$ and $z \in \Tn$, then $z^m\widetilde{F}(z) = \mathbf{0}$ for all exponents $m \in \mathbb{Z}^{\Sigma(1)}$. This implies that $\f \otimes z^m(\widetilde{F})(z) = 0$ for all $\f \in (\mathbb{C}^s)^{\vee}$. Conversely, if $z^{m_j} \otimes \f_j(\widetilde{F})(z) = 0$ for all $j = 1,\dots,s$ as $\f_1,\dots,\f_s$ form a basis of $(\C^s)^{\vee}$, we deduce that $\widetilde{F}(z) = \mathbf{0}$, implying part (i).

We proceed with part (ii). By \eqref{eq: sigmamonomials}, the monomials $z^{\mathbf{V}(b_{\sigma}) + a}$ are invertible in $R_{\sigma}$. Therefore, it is sufficient to show that $\IFk _{\sigma}$ is generated by the polynomials:
\begin{equation}
\label{eq: basisIsigbar}\widetilde{F}_j^{\sigma}\coloneqq  z^{\mathbf{V}(b_{\sigma} - b_j) + a - \omega_s}\f_j(\widetilde{F}), \quad j \in \{1,\dots,s \}.
\end{equation}
For every element $\f \otimes z^{\mathbf{V}b + a - \omega_s} \in H_{\mathcal{E}_F^{\vee}(D_F)}$ for some $b \in \mathbb{Z}^n$, we consider the decomposition
    $$\varphi = \sum_{i = 1}^s\alpha_i\f_i, \quad \alpha_i \in \mathbb{C}. $$
As there exists a basis of $\widetilde{E}_{\rho}^i$ for $\rho \in \sigma(1)$ given by a subset of $\{\f_1,\dots,\f_s\}$, we have that
$$
        -\langle v_{\rho},b \rangle \leq \max(i \in \mathbb{Z} : \: \f \in \widetilde{E}_{\rho}^i) =  \min_{\alpha_j \neq 0} \, \max(i \in \mathbb{Z} : \: \f_j \in \widetilde{E}_{\rho}^i) = \min_{\alpha_j \neq 0} -\langle v_{\rho}, b_{\sigma} - b_j \rangle \quad \forall \rho \in \sigma(1).
$$
Therefore, if $\alpha_j \neq 0$, we have $-\langle v_{\rho},b \rangle \leq -\langle v_{\rho},b_{\sigma} - b_j \rangle$ for all $\rho \in \sigma(1)$ and so
$$z^{\mathbf{V}(b+b_j - b_{\sigma})} =\prod_{\rho \in \Sigma(1)}z_{\rho}^{\langle v_{\rho},b + b_i - b_{\sigma} \rangle} \in R_{\sigma}$$
for all $j \in \{1,\dots,s\}$ with $\alpha_j \neq 0$. This implies that
$$\f \otimes z^{\mathbf{V}b + a - \omega_s}(\widetilde{F}) = \sum_{\alpha_j \neq 0}\alpha_j\f_j \otimes z^{\mathbf{V}b + a - \omega_s}(\widetilde{F}) = \sum_{\alpha_j \neq 0}\big(\alpha_jz^{\mathbf{V}(b + b_i - b_{\sigma})}\big)\widetilde{F}_j^{\sigma},$$
which is a combination in $R_\sigma$. Therefore, $\IFk_{\sigma}$ is generated by the polynomials in \eqref{eq: basisIsigbar}.
\end{proof}

Generalizing the case of sparse polynomial systems, when $s = n$ the ideal $I(F)$ has no solutions at infinity for generic values of the parameters.

\begin{theorem}
\label[theorem]{thm: MainNoSolutionsInfinity}
Let $F = C(u \star x^B)$ be a vertical system and let  $(X_{\Sigma},\mathcal{E}_F)$ be a simplicial vertical pair associated to it. In the case where $s = n$, we also have that for generic values of $u \in \C^m$, $V(\IFk) \cap V(z_{\rho}) = \varnothing$.
\end{theorem}

\begin{proof}
It suffices to show that for all $\rho \in \Sigma(1)$ and $\sigma \in \Sigma(n)$, there exists a family of nonempty open subsets $V_{\rho,\sigma} \subset \mathbb{C}^m$ such that $V(\IFk) \cap V(z_{\rho}) \cap U_{\sigma} = \varnothing$ for $u \in V_{\rho,\sigma}$. As before, we assume that the vertex of $\Delta_F$ corresponding to $\sigma$ is given by the first $s$ columns of $B$.

If $\rho \notin \sigma(1)$, we always have $V(z_{\rho}) \cap U_{\sigma} = \varnothing$ and take $V_{\rho,\sigma} = \mathbb{C}^m$. If $\rho \in \sigma(1)$, part (ii) in \Cref{thm: Torus} implies that $\IFk_{\sigma} + \langle z_{\rho} \rangle$ is generated by the polynomials \eqref{equationFG} after setting the variable $z_{\rho}$ to zero, which are of the form
$$ u_i + \widehat{G}_i(u,z)$$
where $\widehat{G}_i \in R_{\sigma}$ does not depend on the parameters $u_i$ for $i = 1,\dots,s$. Consider the map
\[
\begin{aligned}
\Phi: (U_{\sigma} \cap V(z_{\rho})) \times \mathbb{C}^{m - n} &\to \mathbb{C}^m, \\
\bigl( (z_{\rho'})_{\rho' \in \sigma(1) \setminus \{\rho\}}, u_{s+1},\dots,u_m \bigr) &\mapsto
\bigl( -\widehat{G}_1(u,z),\dots,-\widehat{G}_n(u,z),u_{s+1},\dots,u_m \bigr).
\end{aligned}
\]
The set $V_{\rho,\sigma} \coloneqq \mathbb{C}^m \setminus \overline{\text{Im} \, \Phi}$ is open. Moreover, as $U_{\sigma} \cap V(z_{\rho})$ is an affine variety of dimension $n - 1$, the source has dimension $(n-1)+(m-n)=m-1<m$, so $\overline{\operatorname{Im}\Phi} \neq \mathbb{C}^m$ and hence $V_{\rho,\sigma} \neq \varnothing$. As $\overline{\text{Im} \, \Phi}$ contains the set of parameters for which $V(\IFk) \cap V(z_{\rho}) \neq \varnothing$ in $U_{\sigma}$, we deduce the theorem.
\end{proof}

\section{The generic number of zeros of vertical systems}
\label{sec: KoszulChernZeros}

In the previous section, we have constructed a family of ideals $\IFk$ with the same zeros as the vertical family $F = C(u \star x^B)$ over $\Tn$ and proved that these ideals give overdetermined systems in the faces of the toric variety. In this section, we compute the generic number of zeros of square vertical systems with the following strategy: we first show that if $V(\IFk)$ is finite, the number of zeros of $V(\IFk)$ (counted with multiplicity) in $X_{\Sigma}$ is given by the top Chern class of $\mathcal{E}_F$. Then, we use \Cref{thm: MainNoSolutionsInfinity} to deduce that this equals the generic root count of $F$ in $\Tn$. By exploiting the resolutions in \Cref{thm: resolutions}, we prove \cref{thm: RootCountFinalFormula} and \cref{thm: splitRootCount}, which give formulas for this generic root count.

\subsection{The Koszul complex}

As a first step, we make use of the Koszul complex over sections of the toric vector bundle $\mathcal{E}_F$, showing that the differentials can be described by the polynomials $\widetilde{F}^{\ell}$ defined in \eqref{eq: widetilde}.
Namely, we consider the Koszul complex over the toric vector bundle $\mathcal{E}_F$, i.e.
$$K^{\bullet}(\mathcal{E}_F): \big( 0 \xrightarrow[]{} \wedge^s\mathcal{E}_F^{\vee} \xrightarrow[]{d_s} \wedge^{s-1}\mathcal{E}_F^{\vee} \xrightarrow[]{d_{s-1}} \dots \xrightarrow[]{} \wedge^2\mathcal{E}_F^{\vee} \xrightarrow[]{d_2} \mathcal{E}_F^{\vee} \xrightarrow[]{d_1} \mathcal{O} \big)$$
with differentials defined as
$$ d_{\ell}( (\f_1 \wedge \dots \wedge \f_{\ell}) \otimes z^m) = \sum_{i = 1}^{\ell}(-1)^i z^m\f_i(\widetilde{F}^{(s-\ell + 1)}) \otimes (\f_1 \wedge \dots \wedge \widehat{\f}_i \wedge \dots \wedge \f_{\ell})$$
for each section $(\f_1 \wedge \dots \wedge \f_{\ell}) \otimes z^m$ of $\wedge^{\ell}\mathcal{E}_F^{\vee}$ with $z^{m} \in R^{\pm}$. Using \Cref{prop: maxminwedgeDual}, we deduce that the degree of the map $d_{\ell}$ equals
$$D_{\wedge^{\ell-1}\mathcal{E}_F^{\vee},1} - D_{\wedge^{\ell}\mathcal{E}_F^{\vee},1} = \sum_{j = 1}^{\ell-1}D_{\mathcal{E}_F^{\vee},j} - \sum_{j = 1}^{\ell}D_{\mathcal{E}_F^{\vee},j} = -D_{\mathcal{E}_F^{\vee},\ell} = D_{F,s-\ell + 1},$$
which matches the degree of $\widetilde{F}^{s - \ell + 1}$. We show that, in the open cover $U_{\sigma}$, this complex is isomorphic to the Koszul complex over the generators of $\IF$ in \eqref{eq: Fisigma}.

\begin{proposition}
\label[proposition]{thm: complexIsWellDefined}
Let $F = C(u \star x^B)$ be a vertical system and let $(X_{\Sigma},\mathcal{E}_F)$ be a simplicial vertical pair associated to $F$. Then,
\begin{enumerate}[label = (\roman*)]
    \item The complex $K^{\bullet}(\mathcal{E}_F)$ is well-defined.
    \item For each $\sigma \in \Sigma(n)$, the complex $K^{\bullet}(\mathcal{E}_F)_{\mid U_{\sigma}}$ is isomorphic to the Koszul complex over the generators of $\IF_{\sigma}$ in \eqref{eq: Fisigma}.
\end{enumerate}

\end{proposition}

\begin{proof}
    Using \Cref{rk: nonGlobalSections}, it suffices to prove part (i) for the global sections of a toric vector bundle $\wedge^{\ell}\mathcal{E}_F^{\vee} \otimes \mathcal{O}(D)$ for some $\Tn$-invariant divisor $D$. Thus, part (i) follows from \Cref{lemma: mainHomogenization}. For part (ii), we assume that the vertex of $\Delta_F$ corresponding to $\sigma$ is given by the first $s$ columns of $B$ and denote by $[s]^{\ell}$ the set of subsets of $[s]$ of size $\ell$. For each $J \in [s]^{\ell}$, we denote $b_J \coloneqq \sum_{j \in J}b_j$. From the isomorphism in \eqref{eq: trivialization}, we deduce that
\begin{equation}
\label{trivialization2}
\wedge^{\ell}\mathcal{E}_F^{\vee}{\mid_{U_{\sigma}}} \cong \bigoplus_{J \in [s]^{\ell}} \mathcal{L}_{-b_{J}}{\mid _{U_{\sigma}}}.
\end{equation}
Each summand corresponds to the basis element $\f_J \coloneqq \wedge_{j \in J}\f_j \in \wedge^{\ell}(\mathbb{C}^s)^{\vee}$ where $\{\f_1,\dots,\f_s\}$ is the basis in \eqref{eq: Fisigma}. As we established that the sections of $\wedge^{\ell}\mathcal{E}_F^{\vee}$ have degree $-\sum_{j = 1}^{\ell}D_{F,s-j+1}$, \eqref{trivialization2} induces an isomorphism
$$ \psi_{\ell}  :H^0(U_{\sigma},\wedge^{\ell}\mathcal{E}_F^{\vee}) \xrightarrow[]{} \bigoplus_{J \in [s]^\ell} \f_J \otimes R_{\sigma}, \quad \f_J \otimes z^m  \xrightarrow[]{} \f_J \otimes z^{m + \mathbf{V}(b_{J})+ \lambda_{\ell}}.$$
As a complex of sheaves over an affine variety is determined by its global sections, it is sufficient to show that the diagram

\[
\begin{array}{ccccccccc}
\label{diagram1}
\longrightarrow & H^0(U_{\sigma},\wedge^{\ell}\mathcal{E}_F^{\vee}) & \xrightarrow{d_{\ell}} & H^0(U_{\sigma},\wedge^{\ell-1}\mathcal{E}_F^{\vee}) & \longrightarrow &\\
 &  \downarrow{\small\psi_{\ell}} & & \downarrow{\tiny{\psi_{\ell-1}}} & & \\
 \longrightarrow & \bigoplus_{J \in [s]^\ell} \f_J \otimes R_{\sigma} & \xrightarrow[\widehat{d_{\ell}}]{} & \bigoplus_{J \in [s]^{\ell-1}} \f_J \otimes R_{\sigma} & \longrightarrow
\end{array},
\]

commutes, where the map $\widehat{d}_{\ell}$ equals the Koszul complex of the sequence $F_1^{\sigma},\dots,F_s^{\sigma}$, i.e.

\[
\widehat{d}_{\ell}(\f_J \otimes z^m ) = \sum_{j \in J}(-1)^{\iota(j)} \f_{J \setminus \{j\}} \otimes z^m F_j^{\sigma},
\]
where $\iota(j)$ indicates the position of $j$ in the order of the elements of $J$. If we consider a section $\f_{J} \otimes z^m \in H^0(U_{\sigma},\wedge^{\ell}\mathcal{E}_F^{\vee})$, where $m = \mathbf{V}(b) - \lambda_{\ell}$ for some $b \in \mathbb{Z}^n$, then
\begin{align*}
(\psi_{\ell-1})^{-1}\left(\widehat{d}_{\ell}\left(\psi_{\ell}(\f_J \otimes z^m\right)\right) =  (\psi_{\ell-1})^{-1}\left(\widehat{d}_{\ell}\left(\f_J \otimes z^{m + \mathbf{V}(b_J) + \lambda_{\ell}}\right)\right) = \\
(\psi_{\ell-1})^{-1}\left(\sum_{j \in J}(-1)^{\iota(j)} (\f_{J \setminus \{j\}} \otimes z^{\mathbf{V}(b + b_{J})} F_j^{\sigma}) \right) = \\
(\psi_{\ell-1})^{-1}\left(\sum_{j \in J}(-1)^{\iota(j)}\sum_{k = 1}^m  \f_{J \setminus \{j\}} \otimes \f_j(\gamma_k)u_kz^{\mathbf{V}(b_k + b + b_{J \setminus \{j\}})} \right) =
\end{align*}
\begin{align*}
\sum_{j \in J}(-1)^{\iota(j)}\sum_{k = 1}^m  \f_{J \setminus \{j\}} \otimes \f_j(\gamma_k)u_kz^{\mathbf{V}(b_k + b) - \lambda_{\ell - 1}}  = \\
\sum_{j \in J}(-1)^{\iota(j)} \f_{J \setminus \{j\}} \otimes z^m \f_j(\widetilde{F}^{s - \ell + 1}) =
d_{\ell}\left(\f_J \otimes z^m\right), \hspace{0.2cm}
\end{align*}
concluding that the above diagram commutes and the two complexes are isomorphic.
\end{proof}

As a consequence, if $\dim V(\IFk) = n - s$, then $K^{\bullet}(\mathcal{E}_F)$ is exact.

\begin{corollary}
\label[corollary]{corolary: exact}
Let $F = C(u \star x^B)$ be a vertical system and $(X_{\Sigma},\mathcal{E}_F)$ a simplicial vertical pair associated to $F$. If $V(\IFk)$ has pure dimension $n-s$, then $K^{\bullet}(\mathcal{E}_F)$ is exact.
\end{corollary}

\begin{proof}
By \Cref{thm: Torus}, we have that if $V(\IFk) \cap U_{\sigma}$ has pure dimension $n-s$, then $(\IFk)_{\sigma}$ is a complete intersection for all maximal cones $\sigma \in \Sigma(n)$. As $U_{\sigma}$ is a normal toric variety (in particular, it is Cohen-Macaulay \cite[Theorem 9.2.9]{coxlittleschneck}), $K^{\bullet}(\mathcal{E}_F){\mid _{U_{\sigma}}}$ is exact for all $\sigma \in \Sigma(n)$.
\end{proof}

\subsection{Chern classes and the degree of $\IFk$}

In this section, we study the number of zeros of $\IFk$ in $X_{\Sigma}$ for generic values of $u \in \mathbb{C}^m$ using intersection theory \cite{Fulton1984}.

Similarly to \eqref{def: ideal}, for each $D \in \Div(X_{\Sigma})$, we may define the ideal
$$\IF_D \coloneqq \langle\f \otimes z^m(\widetilde{F}) : \: \f \otimes z^m \in H_{\mathcal{E}_F^{\vee}(D)}\rangle.$$
Before proceeding to the main theorem, we prove the following technical lemma.

\begin{lemma}
\label[lemma]{lemma: generators}
    Let $F = C(u \star x^B)$ be a vertical system and let $(X_{\Sigma},\mathcal{E}_F)$ be a vertical pair associated to it. There exists a nef divisor $D \in \Div(X_{\Sigma})$ such that for all $u \in \mathbb{C}^m$, the subschemes of $X_{\Sigma}$ defined by $\IFk$ and $\IFk_{D' + D_F}$ are isomorphic for all nef divisors $D' \geq D$.
\end{lemma}

\begin{proof}
We consider the surjective map of sheaves
$$H^0(X_{\Sigma},\mathcal{E}_F^{\vee}(D_F)) \otimes \mathcal{O}  \xrightarrow[]{} \mathcal{E}_F^{\vee}(D_F).$$
We denote by $\mathcal{K}$ its kernel. Tensoring with $\mathcal{O}(D)$ and taking the long exact sequence of cohomology, we get
$$\xrightarrow[]{} H^0(X_{\Sigma},\mathcal{E}_F^{\vee}(D_F)) \otimes H^0(X_{\Sigma},\mathcal{O}(D')) \xrightarrow[]{} H^0(X_{\Sigma},\mathcal{E}_F^{\vee}(D_F+D')) \xrightarrow[]{} H^1(X_{\Sigma},\mathcal{K} \otimes \mathcal{O}(D')) \xrightarrow[]{}$$
By Serre's vanishing \cite[Theorem 9.0.3]{coxlittleschneck}, there exists a nef divisor $D$ such that for all $D' \geq D$ the last term vanishes and the map
$$H^0(X_{\Sigma},\mathcal{E}_F^{\vee}(D_F)) \otimes H^0(X_{\Sigma},\mathcal{O}(D')) \xrightarrow[]{} H^0(X_{\Sigma},\mathcal{E}_F^{\vee}(D_F+D')) $$
is surjective. This implies that
$$\IFk_{D_F + D'} = R_{[D']}\cdot\IFk.$$
As $\mathcal{O}(D')$ is globally generated, i.e. it has no base points, we can deduce that the schemes defined by both ideals coincide.
\end{proof}

The main theorem of this section computes the degree of $\IFk$ in terms of the top Chern class of $\mathcal{E}_F$, when the ideal defines a finite subscheme of $X_{\Sigma}$.

\begin{theorem}
\label[theorem]{thm: VerticalHirzebruchRiemannRoch}
Let $F = C(u \star x^B)$ be a vertical system with $s = n$ and $(X_{\Sigma},\mathcal{E}_F)$ a simplicial vertical pair associated to $F$. If $V(\IFk) \subset X_{\Sigma}$ is finite, then
    $$\deg(\IFk) = \int_{X_{\Sigma}} c_n(\mathcal{E}_F).$$
\end{theorem}

\begin{proof}
Let $\mathcal{I}_u \subset \mathcal{O}$ be the ideal sheaf defined by the image of $d_1$ in the Koszul complex at parameter values $u \in \mathbb{C}^m$. Let $D \in \Div(X_{\Sigma})$ satisfy:
\begin{enumerate}[label = (\roman*)]
\label{eq: parts}
    \item $H^i(X_{\Sigma},\wedge^{\ell}\mathcal{E}_F^{\vee} \otimes \mathcal{O}(D)) = 0$ for all $i > 0$ and $\ell \in \{1,\dots,n\}$.
    \item The ideal sheaf $\mathcal{I}_u \otimes \mathcal{O}(D)$ is generated by global sections.
    \item The subschemes of $X_{\Sigma}$ defined by $\IFk$ and $\IFk_D$ are isomorphic.
\end{enumerate}
Nef divisors satisfying these conditions exist due to Serre's vanishing theorem \cite[Theorem 9.0.3]{coxlittleschneck} and our previous \Cref{lemma: generators}. Moreover, the subscheme $Z \subset X_{\Sigma}$ defined by $\mathcal{I}_u$ is isomorphic to the subscheme defined by $\mathcal{I}_u \otimes \mathcal{O}(D)$. From \Cref{corolary: exact}, the $K^{\bullet}(\mathcal{E}_F)$ is exact. By (i), twisting the Koszul complex by $\mathcal{O}(D)$ and taking global sections, we get an exact sequence of vector spaces,
\begin{equation}
\label{eq: globalsections}
0 \xrightarrow[]{} H^0(X_{\Sigma},\wedge^s\mathcal{E}_F^{\vee} \otimes \mathcal{O}(D))  \xrightarrow[]{d_{s}} \dots \xrightarrow[]{d_2} H^0(X_{\Sigma},\mathcal{E}_F^{\vee} \otimes \mathcal{O}(D)) \xrightarrow[]{d_1} H^0(X_{\Sigma},\mathcal{O}(D)).
\end{equation}
In particular, $Z$ is isomorphic to the scheme defined by the image of the last map, which equals the subscheme defined by $\IF_D$. Using (ii) and (iii), $Z$ is a finite subscheme with
$$\deg(\IFk) = \deg(\IFk_{D}) = \deg(Z) = \chi(\mathcal{O}_Z).$$
As the Koszul complex is exact, this equals
$$\sum_{\ell \geq 0}(-1)^{\ell}\chi(\wedge^{\ell}\mathcal{E}_F^{\vee}).$$
Using \Cref{thm: HRR}, this computation is reduced to computing the Chern characters of $\wedge^i\mathcal{E}_F^{\vee}$ for $i \in \{1,\dots,n\}$. If we express the total Chern class of $\mathcal{E}_F$ as
$$c(\mathcal{E}_F) = \prod_{j = 1}^n(1 + x_j),$$
we may write the Chern characters as $$\ch(\wedge^{\ell}\mathcal{E}_F^{\vee}) = \sum_{\substack{J \subseteq [n]\\ |J|=\ell}}e^{-\sum_{j \in J}x_j}, \quad \ell \in \{1,\dots,s\}.$$
With this, we compute the alternating sum of Chern characters
\begin{equation}
\label{eq: alternatesum}
\sum_{\ell = 0}^n(-1)^\ell \ch(\wedge^\ell \mathcal{E}_F^{\vee}) = \sum_{J \subset [n]}(-1)^{|J|}e^{-\sum_{j \in J}x_j} =\prod_{j = 1}^n(1 - e^{-x_j}) = \prod_{j = 1}^n x_j \left( 1 - \frac{x_j}{2} + \frac{x_j^2}{6} - \cdots \right).\end{equation}
The Todd class has its degree $0$ component equal to $1$, so we may write
\begin{equation}
\label{eq: toddsigma}
\Td(X_{\Sigma}) = 1 + \Td_1(X_{\Sigma}) + \Td_2(X_{\Sigma}) + \cdots.
\end{equation}
In particular, the degree-$n$ part of the product of \eqref{eq: alternatesum} and \eqref{eq: toddsigma} is $x_1\cdots x_n$. Therefore, using \Cref{thm: HRR},
$$\sum_{\ell \geq 0}(-1)^{\ell}\chi(\wedge^{\ell}\mathcal{E}_F^{\vee}) = \int_{X_{\Sigma}}\Big(\sum_{\ell = 0}^n(-1)^{\ell} \ch(\wedge^{\ell} \mathcal{E}_F^{\vee})\Big) \cdot \Td(X_{\Sigma}) = \int_{X_{\Sigma}}x_1\cdots x_n = \int_{X_{\Sigma}}c_n(\mathcal{E}_F),$$
concluding the proof.
\end{proof}

\subsection{The generic number of zeros of vertical systems and $V(\IFk)$}
At this point, we distinguish two cases:

\begin{enumerate}[label = (\roman*)]
    \item If the system $F = C(u \star x^B)$ has degenerate zeros for generic $u \in \mathbb{C}^m$, Feliu, Henriksson and Pascual-Escudero \cite[Theorem 3.7]{FELIU2025630} have shown that  $V_{\Tn}(F_u) = \varnothing$ for generic $u \in \mathbb{C}^m$. We say that these vertical systems are \textit{generically inconsistent}.
    \item If the system $F = C(u \star x^B)$ has nondegenerate zeros for generic $u \in \mathbb{C}^m$, the same authors \cite[Theorem 3.7]{FELIU2025630} have shown that there exists an open subset where $V_{\Tn}(F_u)$ is nonempty and zero-dimensional with all the zeros being nondegenerate. In this case, we say that these vertical systems are \textit{generically consistent}.
\end{enumerate}

If $F = C(u \star x^B)$ is generically inconsistent and $s = n$, the generic root count equals zero. Moreover, for the parameters $u \in \mathbb{C}^m$ where $V_{\Tn}(F_u) \neq \varnothing$, \cite[Theorem 3.7]{FELIU2025630} implies that $\dim V_{\Tn}(F_u) > 0$ and it has no isolated points. Since $X_{\Sigma}$ is projective, $V(\IFk)$ cannot be contained entirely in the affine torus $\Tn$. Therefore, $V(\IFk)$ intersects the toric boundary, so $V(\IFk) \cap V(z_{\rho}) \neq \varnothing$ for some $\rho \in \Sigma(1)$, proving part (iii) of \Cref{mainTheoremIntro} for this case. Therefore, we focus on the case where $F$ is generically consistent.

\begin{theorem}
\label[theorem]{cor: numberOfIsolatedSolutions}
Let $F = C(u \star x^B)$ be a generically consistent vertical system with $s = n$ and let $(X_{\Sigma},\mathcal{E}_F)$ be a simplicial vertical pair associated to it. For each $u \in \C^m$, we have
$$|V_0(F_{u}) \cap \Tn| \leq \int_{X_{\Sigma}}c_n(\mathcal{E}_F) = \max_{u' \in \mathbb{C}^m}|V_0(F_{u'}) \cap \Tn| $$
Moreover, the bound is attained if and only if $V(\IFk) \cap V(z_{\rho}) = \varnothing$ for all $\rho \in \Sigma(1)$, which happens in an open subset of $\mathbb{C}^m$.
\end{theorem}

\begin{proof}
By \cite[Corollary 2.16]{genericrootcounts}, there exists a nonempty open subset $U_0 \subset \mathbb{C}^m$ such that for every $u \in U_0$, the number of isolated zeros of $F_u(z) = 0$ in $\Tn$, counted with multiplicity, equals the maximal number of isolated zeros in $\Tn$ over all parameters. If $V(\IFk)$ is not finite, then since $X_{\Sigma}$ is projective, $V(\IFk)$ cannot be contained entirely in the affine torus $\Tn$, i.e. $V(\IFk) \cap V(z_\rho)$ is nonempty for some $\rho \in \Sigma(1)$. Then, the same argument as in \cite[Theorem B, b)]{Bernshtein1975TheNO} yields that the maximal number of isolated zeros in $\Tn$ cannot be attained.

When $V(\IFk)$ is finite, \Cref{thm: Torus} and \Cref{thm: VerticalHirzebruchRiemannRoch} imply that
\[
|V_{\Tn}(F_u)| = |V_{\Tn}(\IFk)| \leq \int_{X_{\Sigma}} c_n(\mathcal{E}_F),
\]
where the number of solutions is counted with multiplicity. By \cite[Theorem 3.7]{FELIU2025630}, there exists an open subset $U_1 \subset \mathbb{C}^m$ such that $\dim V_{\Tn}(F_u) = 0$ and $V_{\Tn}(F_u) \neq \varnothing$ for all $u \in U_1$. Additionally, \Cref{thm: MainNoSolutionsInfinity} ensures the existence of an open subset $U_2 \subset \mathbb{C}^m$ such that
\[
V(\IFk) \cap V(z_\rho) = \varnothing
\]
for all rays $\rho \in \Sigma(1)$. Thus, $u \in U_0 \cap U_1 \cap U_2$ if and only if $V(\IFk) \cap V(z_\rho) = \varnothing$ for all $\rho \in \Sigma(1)$. Applying \Cref{thm: VerticalHirzebruchRiemannRoch} to this open set yields
\[
|V_0(F_u) \cap \Tn| = |V(\IFk)| = \int_{X_{\Sigma}} c_n(\mathcal{E}_F),
\]
if and only if $u \in U_0 \cap U_1 \cap U_2$.
\end{proof}

\subsection{Generic root counts in terms of mixed volumes}
In \Cref{thm: VerticalHirzebruchRiemannRoch}, we have reduced finding the degree of $\IFk$ to computing the top Chern class of $\mathcal{E}_F$. In order to deduce the value of $c_n(\mathcal{E}_F)$, we will use the resolutions coming from the Weil decoration of a toric vector bundle (see \Cref{thm: resolutions}). Recall that, given any poset $(\mathcal{S},\leq)$, the M\"obius function $\mu: \mathcal{S} \times \mathcal{S} \xrightarrow[]{}\mathbb{Z}$ is recursively defined as
$$\mu(x,y) \coloneqq \begin{cases}
    1 & x = y \\
    -\sum_{x \leq z < y}\mu(x,z) & x \neq y.
\end{cases}$$

\begin{definition}
\label[definition]{def: almostNeflyDecorated}
Let $X_{\Sigma}$ be a complete toric variety and let $\mathcal{E}$ be a toric vector bundle of rank $s$. For each $D \in \mathcal{S}_E$, we define
$$\delta_{D} \coloneqq \sum_{D \geq D' \in \mathcal{S}_E}\dim(V_{D'})\mu(D',D),$$
where $\mu$ is the Möbius function on $\mathcal{S}_E$.
\end{definition}

Firstly, we compute the total Chern class of $\mathcal{E}_F$ in terms of the line bundles in $\mathcal{S}_E$.

\begin{lemma}
\label[lemma]{lemma: totalChernClass}
Let $F = C(u \star x^B)$ be a vertical system and let $(X_{\Sigma},\mathcal{E}_F)$ be a vertical pair. Then, the total Chern class of $\mathcal{E}_F$ equals
\begin{equation}
\label{rk: delta_D}
c(\mathcal{E}_F) = \prod_{D \in \mathcal{S}_E}(1 + c_1(\mathcal{O}(D)))^{\delta_D}.
\end{equation}
\end{lemma}

\begin{proof}
By \cite[Proposition 3.1.11]{rotaway}, we have that for all $D,D' \in \mathcal{S}_E$
$$\mu(D,D') = \sum_{k \geq 0}(-1)^k|\mathcal{C}_k(D,D')|,$$
i.e. the M\"obius function equals the alternating sum over the chains of length $k$ in the poset. Using the resolution in \Cref{thm: resolutions} and recalling the formula \eqref{eq: whitney} for the total Chern class, we have that, for fixed $D' \in \mathcal{S}_E$, the alternating sum over all terms of the form $V_{D} \otimes \mathcal{O}(D')$ for some $D \leq D'$ yields
$$ \prod_{D' \geq D \in \mathcal{S}_E} c(V_{D} \otimes \mathcal{O}(D'))^{\sum_{k = 0}^s(-1)^k|\mathcal{C}_k(D,D')|}.$$
Moreover, we have $c(\mathcal{O}(D')) = 1 + c_1(\mathcal{O}(D'))$ and $c(V_{D} \otimes \mathcal{O}(D')) = c(\mathcal{O}(D'))^{\dim V_{D}}$, implying the formula.
\end{proof}

We assume that $\mathcal{S}_E = \{D_1,\dots,D_t\}$ for some $t \geq 0$. For each $D_i \in \mathcal{S}_E$, we denote by $\Delta_i$ the corresponding polytope and $\delta_i$ the exponent in \eqref{rk: delta_D}. If $p \in \mathbb{Z}^t_{\geq 0}$ is such that $\sum_{i = 1}^t p_i = n$, we denote by $\MV(\Delta_1[p_1],\dots,\Delta_t[p_t])$, the mixed volume with $p_i$ copies of $\Delta_i$ for $i \in \{1,\dots,t\}$. We recall that if $\delta < 0$, then
$$\binom{\delta}{p} \coloneqq (-1)^p\binom{p - \delta - 1}{p}.$$

Our final main result describes the generic root count over $\Tn$ for the vertical system $F$ as an alternating sum of the mixed volumes of the polytopes associated to the divisors $D \in \mathcal{S}_E$.

\begin{theorem}
\label[theorem]{thm: RootCountFinalFormula}
Let $F = C(u \star x^B)$ be a generically consistent vertical system with $s = n$ and let $(X_{\Sigma},\mathcal{E}_F)$ be a simplicial vertical pair. If $D$ is nef for all $D \in \mathcal{S}_E$, then for all $u \in \C^m$, we have
$$|V_0(F_{u}) \cap \Tn| \leq \sum_{\substack{p \in \mathbb{Z}_{\geq0}^{t} \\
p_1 + \dots + p_t = n}} \;
\mu_{p}\MV\big( \Delta_1[p_1],\dots,\Delta_t[p_t]\big), \quad \mu_{p} \coloneqq
    \prod_{i = 1}^t\binom{\delta_i}{p_i}, $$
where $\Delta_1,\dots,\Delta_t \subset \mathbb{R}^n$ are the polytopes associated with the divisors $D_1,\dots,D_t$. Moreover, there exists a Zariski open subset $U \subset \mathbb{C}^m$ where the bound is attained.
\end{theorem}

\begin{proof}
If we let $x_i = c_1(\mathcal{O}(D_i))$ for $i \in [t]$, we have to consider the degree $n$ part of the series
$\prod_{i = 1}^t(1 + x_i)^{\delta_i}$. Recall that
$$(1 + x_i)^{\delta_i} = \sum_{p \geq 0}\binom{\delta_i}{p}x_i^p.$$
Taking the product over all of the power series, we have that
\begin{equation}
\label{eq: expansion}
\prod_{i = 1}^t(1 + x_i)^{\delta_i} = \sum_{p \in \mathbb{Z}_{\geq 0}^t}\mu_p \prod_{i = 1}^t x_i^{p_i}.
\end{equation}
As all of the $D \in \mathcal{S}_E$ that appear in the expression \eqref{rk: delta_D} are nef,  \Cref{thm: MixedVolumeChernClass} implies that for all $p \in \mathbb{Z}_{\geq 0}^t$ such that $\sum p_i = n$, we have
$$\int_{X_{\Sigma}} \prod_{i = 1}^t x_i^{p_i} = \MV(\Delta_1[p_1],\dots,\Delta_t[p_t]).$$
Taking the degree $n$ part in \eqref{eq: expansion} and using \cref{thm: VerticalHirzebruchRiemannRoch}, we deduce the result.
\end{proof}

\begin{remark}
\label[remark]{rk: notNeflyDecorated}
     The root count can also be expressed in terms of mixed volumes when $\mathcal{E}_F$ is not nef, as long as the divisors $D \in \mathcal{S}_E$ are written as the difference of two nef divisors. Namely, if $D = D_1 - D_2$, for two nef divisors $D_1,D_2 \in \Nef(X_{\Sigma})$, then
     $$c(\mathcal{O}(D)) = 1 + c_1(\mathcal{O}(D_1)) - c_1(\mathcal{O}(D_2)).$$
     In that case, we consider $\Delta_D$ to be a difference of polytopes $\Delta_{D_1} - \Delta_{D_2}$ and extend the mixed volume multilinearly.
\end{remark}

\begin{theorem}
\label[theorem]{thm: splitRootCount}
Let $F = C(u \star x^B)$ be a generically consistent vertical system and $(X_{\Sigma},\mathcal{E}_F)$ a simplicial vertical pair associated to $F$ and let $u \in \mathbb{C}^m$. If $s = n$ and $\mathcal{E}_F = \bigoplus_{i = 1}^n \mathcal{O}(D_i)$ for some toric line bundles $D_i$, then
$$|V_0(F_{u}) \cap \Tn| \leq \MV(\Delta_1,\dots,\Delta_n), $$
where $\Delta_i$ are the polytopes associated with $D_i$ for each $i \in \{1,\dots,s\}$ (see \eqref{eq:polytope}). Moreover, there exists a Zariski open subset $U \subset \mathbb{C}^m$ where the bound is attained.
\end{theorem}

\begin{proof}
    Follows from \Cref{thm: MixedVolumeChernClass} and \Cref{cor: numberOfIsolatedSolutions}.
\end{proof}

In the case where $\mathcal{E}_F$ splits, the divisors $D_i$ appearing in \Cref{thm: splitRootCount} are nef.

\begin{corollary}
\label[corollary]{cor: neflysplit}
Let $F = C(u \star x^B)$ be a vertical system and $(X_{\Sigma},\mathcal{E}_F)$ a simplicial vertical pair associated to $F$ and let $u \in \mathbb{C}^m$. If $s = n$ and $\mathcal{E}_F = \bigoplus_{i = 1}^n \mathcal{O}(D_i)$ for some divisors $D_i \in \Div(X_{\Sigma})$. Then, the divisor $D_i$ is nef for all $i \in \{1,\dots,n\}$.
\end{corollary}

\begin{proof}
    As $\mathcal{E}_F$ splits as a direct sum, the terms in the Klyachko filtration also split. This implies that there exists a basis $\{e_1,\dots,e_n\}$ such that the vector spaces $E_{\rho}^i$ are generated by a subset of it for all $\rho \in \Sigma(1)$ and $i \in \mathbb{Z}$. In particular, if $e_j \notin E_{\rho}^i$, then $E_{\rho}^i \subset \spann\langle e_1,\dots,\widehat{e_j},\dots,e_n \rangle$. Similarly to \Cref{lemm: MinkowskiReason}, we consider a maximal cone $\sigma\in \Sigma(n)$ and the polyhedra
    $$P_i^{\sigma} \coloneqq \{b \in \mathbb{Z}^n : \: \langle v_{\rho},b\rangle \geq -\max(i \in \mathbb{Z} : \: e_i \in E_{\rho}^i), \: \forall \rho \in \sigma(1)\}.$$
Using \eqref{eq: wedgeFiltration}, the sum of the polyhedra $P_i^{\sigma}$ equals the tangent cone of $\Delta_F$ at the vertex corresponding to $\sigma$. Therefore, we may assume that the first $n$ columns of $B$ ($b_1,\dots,b_n \in \mathbb{Z}^n$) are the vertices of $P_1^{\sigma},\dots,P_n^{\sigma}$.  For these lattice points, we must have that $e_i \in E_{\rho}^{-\langle v_{\rho},b_i \rangle}$ for all $\rho \in \sigma(1)$. Using \Cref{remark: reparametrizing} and that the $E_{\rho}^i$ are spanned by subsets of $\{e_1,\dots,e_n\}$, we  may assume that the column of $C$ corresponding to $b_i$ equals the vector $e_i$. In particular, for any other $\rho \in \Sigma(1)$, we have
$$ \langle v_{\rho},b_i \rangle \geq -\max (i \in \mathbb{Z} : \: e_i \in E_{\rho}^i),$$
implying that $b_i \in \Delta_{\mathcal{E}_F}(e_i)$. Applying \Cref{prop: criterionNef}, we deduce that $D_i$ is nef for all $i = 1,\dots,n$.
\end{proof}

\begin{remark}
\Cref{cor: neflysplit} also implies that the divisor $D_F = \sum_{i = 1}^nD_i$ is nef. In particular $\mathcal{E}_F^{\vee}(D_F) = \oplus_{i = 1}^n\mathcal{O}(D_F - D_i)$, where each of the divisors $D_F - D_i$ is also nef. In order to show that the divisors $D_i$ are nef, we only needed to use the coefficients of $F = C(u \star x^B)$ corresponding to the vertices of the associated polytopes. Therefore, we only need to assume algebraic independence of the coefficients that generate these vertices. This is a similar situation to the extremal-genericity studied in \cite{BENDER2024156}.
\end{remark}

\begin{example}
\label[example]{ex: continueTangentBundle}
We continue with \Cref{ex: TangentExample}. Computing from the Klyachko filtration in \eqref{ex: KlyachkoOfTangent}, we deduce that
$$D_{F,1} = D_0 + D_1 + D_2, \quad D_{F,2} = 0,$$
where $D_0,D_1,D_2$ are the three $\Tn$ divisors of $\P^2$ associated to the variables $x,y,z$ in the coordinate ring $\mathbb{C}[x,y,z]$ of $\mathbb{P}^2$. These are all nef line bundles in $\mathbb{P}^2$. Homogenizing $F = C(u \star x^B)$ with respect to $D_{F,2}$, we get a family of Laurent polynomials $\widetilde{F} \in A \otimes R^{\pm}$ which can be written as:
\begin{equation}
\label{eq: exampleHomogenizedSystem}
\widetilde{F} = \begin{cases}
u_1 x^{-1} y + u_2 x^{-1}z  \hspace{3.1cm}- u_5 xz^{-1} - u_6yz^{-1}= 0,\\[4pt]
\hspace{3.1cm} u_3xy^{-1} + u_4 y^{-1}z - u_5 xz^{-1} - u_6yz^{-1} = 0.
\end{cases}
\end{equation}
Note that the Laurent polynomials in \eqref{eq: exampleHomogenizedSystem} are homogeneous of degree $0$. Moreover, we deduce that
$$D_{\mathcal{E}_F}(e) = \begin{cases}
D_0& e = -e_1 - e_2\\
D_1& e = e_1\\
D_2& e = e_2\\
0 & \text{ otherwise.} \end{cases}$$
We consider $D_F = D_0 + D_1 + D_2$, which is nef and satisfies that $D_F - D$ is nef for all $D \in \mathcal{S}_E$. The ideal $\IF \subset A \otimes R$ defined in \eqref{def: ideal} equals
\begin{multline}\IFk = \Big\langle \varphi_1 \otimes yz^2(\widetilde{F}_u), \varphi_1 \otimes y^2z(\widetilde{F}_u), \varphi_2 \otimes xz^2(\widetilde{F}_u), \varphi_2 \otimes x^2z(\widetilde{F}_u), \\
(\varphi_2 - \varphi_1) \otimes x^2y(\widetilde{F}_u), (\varphi_2 - \varphi_1)  \otimes xy^2(\widetilde{F}_u), \varphi_2 \otimes xyz(\widetilde{F}_u), \varphi_1 \otimes xyz(\widetilde{F}_u)  \Big\rangle =  \\
\Big \langle u_3xz^2 + u_4z^3 - u_5xyz - u_6y^2z, u_3xyz + u_4yz^2 - u_5xy^2 - u_6y^3, u_1yz^2 + u_2z^3 - u_5x^2z - u_6xyz \\
u_1xyz + u_2xz^2 - u_5x^3 - u_6x^2y, u_1xy^2 + u_2xyz - u_3x^3 - u_4x^2z, u_1y^3+u_2y^2z - u_3x^2y - u_4xyz, \\
u_3x^2z + u_4xz^2 - u_5x^2y - u_6xy^2, u_1y^2z + u_2yz^2 - u_5x^2y - u_6xy^2\Big\rangle \subset  R_3
\end{multline}
for each $u \in \mathbb{C}^6$. If we consider the restriction to the face $V(z) \subset \mathbb{P}^2$, we get the ideal
\begin{multline*}
    \IF + \langle z \rangle = \Big \langle- u_5xy^2 - u_6y^3, - u_5x^3 - u_6x^2y, u_1xy^2  - u_3x^3, u_1y^3- u_3x^2y, \\
- u_5x^2y - u_6xy^2, - u_5x^2y - u_6xy^2\Big\rangle.
\end{multline*}
Saturating with respect to $\langle x,y \rangle$, we derive that
$$V(\IFk) \cap V(z) = V(-u_5x - u_6y, u_1y^2  - u_3x^2) \subset \mathbb{P}^1 \cong V(z).$$
which has a solution of the form $(x:y:0) \in \mathbb{P}^2$, if and only if $u_1u_5^2 - u_3u_6^2 = 0$. Similar ideals can be derived for the other two faces of $\mathbb{P}^2$, leading to the conditions $u_2u_4^2 - u_5u_3^2 = 0$ and $u_4u_2^2 - u_6u_1^2 = 0$. Therefore,
the generic number of solutions  of the system $F = C(u \star x^B)$ (counted with multiplicity) is attained if and only if
\begin{equation}
\label{eq: uconditions}
u_1u_5^2 - u_3u_6^2 \neq 0, \: u_2u_3^2 - u_5u_4^2 \neq 0, \: u_4u_1^2 - u_6u_2^2 \neq 0.
\end{equation}

      \begin{figure}[t]
    \begin{tikzpicture}[x=0.75cm, y=0.75cm, line width=1.25pt]

      \foreach \x in {-1,1} {
        \draw[color=white!40!black] (\x, 2pt) -- (\x, -2pt);}
      \foreach \y in {-1,1} {
        \draw[color=white!40!black] (2pt, \y) -- (-2pt, \y);}
      \draw[color=white!40!black, <->] (-2.4, 0) -- (2.4,0) {};
      \draw[color=white!40!black, <->] (0, -1.6) -- (0, 1.8) {};

      \draw[color=blue, fill=blue, opacity=1.0] (0,0) -- (0,-1) -- (1,-1) --
      (0,0) -- cycle;
      \draw[color=blue, fill=blue, opacity=1.5] (0,0) -- (1,0) -- (0,1) --
      (0,0) -- cycle;
      \draw[color=blue, fill=blue, opacity=1.5] (0,0) -- (-1,0) -- (-1,1) --
      (0,0) -- cycle;
      \draw[color=blue, opacity=0.5] (1,0) -- (-1,0) -- (0,-1) -- (0,1) -- (-1,1) -- (1,-1)
      -- (1,0) -- cycle;

      \node[circle, fill=black, inner sep=2.0pt] () at (0,0) {};
      \node[circle, fill=black, inner sep=2.0pt] () at (1,0) {};
      \node[circle, fill=black, inner sep=2.0pt] () at (0,1) {};
      \node[circle, fill=black, inner sep=2.0pt] () at (-1,1) {};
      \node[circle, fill=black, inner sep=2.0pt] () at (-1,0) {};
      \node[circle, fill=black, inner sep=2.0pt] () at (0,-1) {};
      \node[circle, fill=black, inner sep=2.0pt] () at (1,-1) {};
      \node[color=black] () at (-2.4, 0.9) {$D_{\mathcal{E}_F}(e_1)$};
      \node[color=black] () at (2.4, 1.2) {$D_{\mathcal{E}_F}(-e_1-e_2)$};
      \node[color=black] () at (2.4, -0.9) {$D_{\mathcal{E}_F}(e_2)$};
    \end{tikzpicture}
    \caption{To the left, the polytope associated to the vertical with the three divisors appearing in $\mathcal{S}_E$, marked in blue. To the right, the poset defined by $\mathcal{S}_E$.}
    \label{fig:tangentBundle}
\end{figure}
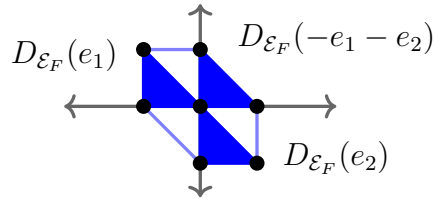

With regard to the number of zeros, we can see that the poset $\mathcal{S}_E = \{D_0,D_1,D_2,0\}$ is defined by
$$D \leq D' \iff D = 0.$$
We also deduce that
$$\delta_D = \begin{cases}
   -1 & D \in \{D_0,D_1,D_2\} \\
   1 & \text{otherwise.}
\end{cases}$$
The polytope associated with $D = 0$ is the lattice point $0 \in \mathbb{Z}^2$, while the other three, denoted as $\Delta_0,\Delta_1,\Delta_2$ are scaled copies of the unit simplex. Applying \Cref{thm: RootCountFinalFormula}, we deduce that if $V(F_u) \cap \Tn$ is finite, then
$$|V(F_u) \cap \Tn| \leq \MV(\Delta_0,\Delta_1) + \MV(\Delta_0,\Delta_2) + \MV(\Delta_1,\Delta_2) = 3.$$
for all $u \in \mathbb{C}^m$, where the bound is attained in the open subset determined in \eqref{eq: uconditions}.
\end{example}

\def\bibfont{\normalsize}
\newcommand{\etalchar}[1]{$^{#1}$}

\Addresses

\end{document}